\documentclass [12pt,a4paper,reqno]{amsart}
 \usepackage{amsmath,amscd,amssymb}
 \usepackage{graphicx}
 \makeatletter
\renewcommand*\env@matrix[1][*\c@MaxMatrixCols c]{%
  \hskip -\arraycolsep
  \let\@ifnextchar\new@ifnextchar
  \array{#1}}
\makeatother
\def\dispace{\setlength{\itemsep}{2pt}}

 \input xy
\xyoption{all}

\newcommand{\longright}[1]{\;{\count255=0 \loop \relbar\mathrel{\mkern-6mu}%
    \advance\count255 by1\ifnum\count255<#1\repeat\rightarrow}\;}
\newcommand{\Right}[2]{\overset{#2}{\longright{#1}}}

\newcommand\isoto{\xrightarrow{
   \,\smash{\raisebox{-0.45ex}{\ensuremath{\scriptstyle\sim}}}\,}}
\newcommand{\Isoto}{\Right{1}{\,\smash{\raisebox{-0.45ex}{\ensuremath{\scriptstyle\sim}}}\,}}

\newcommand{\etype}[1]{\renewcommand{\labelenumi}{(#1{enumi})}}
\def\eroman{\etype{\roman} \dispace}
\def\ealph{\etype{\alph} \dispace}

\def\pSkip{\vskip 1.5mm \noindent}
\newcommand{\ds}[1]{\, {#1} \, }
\newcommand{\dss}[1]{\quad {#1} \quad }

\def\sm{\setminus}
\def\00{ \{ 0 \}}

\def\To{\longrightarrow}

\def\veps{\varepsilon}
\def\vrp{\varphi}

\def\X1{X_1}
\def\Y1{Y_1}

\def\tlb{\tilde b}

\def\tlV{\widetilde V}
\def\tlpsi{\widetilde \psi}
\def\tlrho{\widetilde \rho}

\def\join{\mathrel{\mkern-12mu}}
\def\ont{{\operatorname{t}}}
\newcommand{\trn}[1]{\, {{}^\ont} \join {#1}}

\def\old{{\operatorname{old}}}
\def\new{{\operatorname{new}}}
\def\LIN{{\operatorname{LIN}}}
\def\ELIN{{\operatorname{ELIN}}}
\def\PD{{\operatorname{PD}}}
\def\EPD{{\operatorname{EPD}}}

\def\R{\mathbb R}
\def\BB{\mathbb B}

\def\htb{\hat b}
\def\htq{\hat q}

\def\al{\alpha}
\def\bt{\beta}
\def\gm{\gamma}

\def\sig{\sigma}

\newtheorem{thm}{Theorem} [section]
\newtheorem*{thm*}{Theorem}
\newtheorem{cor}[thm]{Corollary}
\newtheorem{lem}[thm]{Lemma}
\newtheorem{lemma}[thm]{Lemma}
\newtheorem{prop}[thm]{Proposition}

\newtheorem*{claim*} {Claim}
\newtheorem*{theorem4.5'} {Theorem 4.5$'$}
\newtheorem{acknowledgment*}[thm] {Acknowledgment}
\newtheorem{example}[thm]{Example}
\newtheorem{examp}[thm]{Example}
\newtheorem*{examp*}{Example}

 \newtheorem{remark}[thm]{Remark}
  \newtheorem{remarks}[thm]{Remarks}
 \newtheorem*{remark*}{Remark}
 \newtheorem{defn}[thm]{Definition}

\newtheorem{schol}[thm]{Scholium}
\newtheorem{notation}[thm]{Notation}

\newtheorem{problem}[thm]{Problem}
\newtheorem*{notation*} {Notation}
\newtheorem*{notations*} {Notations}

\theoremstyle{remark}
\newtheorem*{caution*} {Caution}
\newtheorem*{comment*} {Comment}
\newtheorem{comment}[thm] {Comment}

\newcommand{\thmref}[1]{Theorem~\ref{#1}}
\newcommand{\propref}[1]{Proposition~\ref{#1}}

\newcommand{\lemref}[1]{Lemma~\ref{#1}}
\newcommand{\corref}[1]{Corollary~\ref{#1}}

 \renewcommand{\sectionmark}[1]{}

\newcommand{\bfem}[1]{\textbf{#1}}

\newcommand{\lm}{\lambda}

 \newcommand{\id}{\operatorname{id}}

 \newcommand{\Hom}{{\operatorname{Hom}}}

\begin{document}

\title[Subordinate Quadratic Forms and Isometric Maps  Over Semirings]
{Subordinate Quadratic Forms \\[2mm] and Isometric Maps\\[2mm] Over Semirings}
\author[Z. Izhakian]{Zur Izhakian}

\author[M. Knebusch]{Manfred Knebusch}









\begin{abstract}
The paper expands the theory of quadratic forms on modules over a semiring~ $R$, introduced in \cite{QF1}--\cite{VR1}, especially in the setup of tropical and supertropical algebra. Isometric linear maps induce subordination on quadratic forms, and provide a main tool in our current study. These maps allow  lifts and pushdowns of quadratic forms on different modules, preserving basic  characteristic properties.
\end{abstract}

\maketitle

\tableofcontents

\numberwithin{equation}{section}

\section*{Introduction}

The theory of quadratic forms over rings and fields has been developing consistently along the past century as an important area of study in classical mathematics. The theory of quadratic forms over semirings is still in its nascent stage, developed in the past decade, and has been motivated by two interacting areas of study: tropical mathematics and real algebra. Semirings arise in several mathematical contexts, in particular in valuation theory.

The paper continues the development of a theory of quadratic forms on modules over a semiring $R$, initiated jointly with Louis Rowen in \cite{QF1}--\cite{VR1}, in the case that $R$ is supertropical.
The basic setup of quadratic forms on an $R$-module $V$, where $R$ is any semiring, has already been given in  \cite{QF1}; \S\ref{sec:1} recalls and complements this setup.

We now give  a brief overview of the supertropical algebra leading to this type of quadratic  form theory,  and then outline  the content of the paper.

  A (commutative) \bfem{semiring} is a set~$R = (R,+, \cdot \; )$ equipped with addition and multiplication,
  such that both $(R,+,0)$ and $(R,\cdot \;,1)$ are abelian monoids
with elements $0=0_R$ and $1=1_R$ respectively, and multiplication
distributes over addition in the usual way.
In other words, $R$ satisfies all the properties of a commutative
ring except the existence of negation under addition. $R$ is a  \textbf{bipotent semiring},  if $a+b \in \{ a, b\} $ for all $a,b \in R$. A semiring
 $R$ is a \bfem{semifield}, if every nonzero element of $R$ is invertible;
  i.e., $R\setminus\{0\}$ is an abelian group.
Main examples of bipotent semifields are the max-plus algebra $(\R \cup \{-\infty\}, \max, + ) $ and the boolean algebra $(\BB, \vee, \wedge)$.
Important (non-tropical) semirings,  where our theory is also applicable, include structures in real algebra, such as the positive cone of an ordered field
\cite[p.~18]{BCR} or a partially ordered commutative ring
\cite[p.~32]{Br}.


Supertropical semrings have a richer algebraic structure \cite{IKR1,IKR2,IzhakianRowen2007SuperTropical}.
A semiring $R$ is a
\textbf{supertropical semiring}, if the following  hold:
\begin{itemize} \dispace
\item $e:= 1_R + 1_R$ is idempotent (i.e., $e + e = E$),
\item the \textbf{ghost ideal} $ e R$ is a bipotent
semiring (and thus totaly ordered), \item  addition is induced by  the ordering of $eR$ and the  \textbf{ghost
map} $a \mapsto ea$:
\begin{equation}\label{eq:0.6}
   a + b = \left\{
             \begin{array}{ll}
               a  & \hbox{if } ea < eb;\\
               b & \hbox{if } eb < ea;\\
               ea  & \hbox{if } ea = eb.
             \end{array}
           \right.
\end{equation}
\end{itemize}
The elements of $eR$ are called \textbf{ghost
elements} and those of $R \sm eR$ are called \textbf{tangible
elements}. The zero element is regarded both as tangible and ghost.

Semirings appear as the targets of non-archimedean valuations $v: \mathbb K \to R$, applied to the field $\mathbb K$  of
\emph{Puiseux series} over  an   algebraically closed field  $F$
of characteristic~$0$. Such a target  can be viewed as a ``max-plus'' semiring, whose operations are ``$+$'' for
multiplication and ``$\sup$'' for addition.  These valuations, called
\textbf{tropicalization}, are described in \cite{IMS}.
The merit of tropicalization is the translation of algebraic and geometric objects into combinatorial entities, which makes their analysis simpler, for example in counting problems in enumerative geometry.
Supertropicalization generalizes tropicalization by taking $R$ to be a supertropical semiring.

Tropicalization applies naturally to basic structures in linear algebra (such as polynomials, matrices, and quadratic forms) simply by replacing the classical addition and
multiplication by maximum and plus respectively. However,  since addition (max) in semirings
does not have negatives, classical theory does not always work well, and one has to take a different route. Supertropicalization proposes a modification of tropicalization nearer to classical arguments.

Bilinear and quadratic forms over semirings, defined in the familiar way (cf. \S\ref{sec:1}), are one example where such discrepancy arises \cite{IKR-LinAlg}.
As in the classical theory,
 these structures are defined on (semi)modules over
  a  semiring ~$R$,  more specifically over  a semifield, when one wants to encode valuable
  ``trigonometric'' information~ \cite{QF2}.


 Recall from \cite[\S1-\S4]{QF1} that a \bfem{module} $V$ over $R$
(also called a \bfem{semimodule}) is an abelian monoid
$(V,+,0_V)$ equipped with a scalar multiplication $R\times V\to V,$
$(a,v)\mapsto av,$ where the same customary axioms of modules over a ring hold: $a_1(bv)=(a_1 b)v,$
$a_1(v+w)=a_1v+a_1w,$ $(a_1+a_2)v=a_1v+a_2v,$ $1_R\cdot v=v,$ and
$0_R\cdot v=0_V=a_1\cdot 0_V$ for all $a_1,a_2,b\in R,$ $v,w\in V.$
We write~$0$ for both $0_V$ and~$0_R$, and 1 for $1_R.$

There are several
notions  of ``basis'' for modules over semifields, since dependence and spanning do not coincide \cite{IKR-LinAlg}. In this paper the
standard categorical notion is used, i.e., an $R$-module $V$ is
\bfem{free}, if it  contains a family $(\veps_i \ds|i\in I)$
such that every $x\in V$ has a unique presentation
$x=\sum\limits_{i\in I} x_i\veps_i$ with scalars $x_i\in R$ and only
finitely many $x_i $ are nonzero.   Then, the family   $(\veps_i \ds |i\in I)$  is called a
\bfem{basis} of $V$. Clearly, any free module with a basis of
$n$ elements is  isomorphic to~$R^{n}$ under the map
$\sum\limits_{i=1}^n x_i\veps_i\mapsto (x_1, \dots, x_n).$

Supertropical semirings establish  a class of semirings over which every free module has a
unique basis (up to multiplication by scalar units), while supervaluations allow to pass from a classical quadratic form on a free module over a
 (commutative) ring  to  quadratic forms on free modules over a supertropical
semiring. In this way,  quadratic forms over semirings provide a new tool to explore families of classical quadratic forms in combinatorial manner.
The present paper develops further the theory of quadratic forms over semirings, enhancing  this tool,  where supertropical semirings serve as main examples.

At first glance, the theory of supertropical quadratic forms on $R$-module $V$ over a supertropical semiring $R$, as described in \cite{QF1}, looks rather different from the classical quadratic forms theory on modules over a ring.
But it seems important  (in particular for a better understanding of the supertropicalization processes on classical quadratic forms) to bridge this gap by establishing a general theory of quadratic forms over semirings.
\{Already supertropical algebra and geometry in general had been driven by the  desire to bring tropical mathematics nearer to the classical arguments.\}

This paper presents one facet of such a theory relying on the study
of \textbf{isometric $R$-linear maps} $\rho: V \to V'$ for given quadratic forms $q: V \to R$ and $q' : V' \to R$ on $R$-modules~ $V$ and $V'$. If, say, $R$ is a field, in good cases (e.g., when $q$ is anisotropic on $\rho(V)$) we have a sort of orthogonal component to $\rho(V)$ in $V'$ (cf. \cite[\S7]{EKM}) and \cite{Spez} for the case that $R$ is of characteristic $2$). But for a semifield $R$,  and all the more for  a semiring $R$, one needs a broader context. We admit modules $V, V'$ over different semirings $R, R'$, connected by a linear map $\lm: R \to R'$, and consider \textbf{$\lm$-isometries}
(Definition \ref{defn:5.1}).

We say that a quadratic form  $q$ is \textbf{subordinate} to a quadratic form  $q'$ with respect to $\lm$  (written  $q\prec _\lm q'$), if there exists a $\lm$-isometry $\rho$ from $(V,q)$ onto the submodule
$(\rho(V), q'|\rho(V))$ of $(V',q')$ (cf. Definition \ref{defn:5.7}). A~ central problem in the paper is to determine for given quadratic forms $q$ and $q'$, whether~ $q$ is subordinate to $q'$ or not.

We further address the  problem for a given map $\lm:R \to R'$ as above, in order to \textbf{lift} a quadratic form $q': V' \to R'$ to a quadratic form $q : V \to R$ by using
a $\lm$-isometry $\rho:V \to V'$ in a natural way, and conversely to \textbf{pushdown} a quadratic form $q': V' \to R'$ by a surjective $\lm$-isometry $\rho: (V',q') \to (V,q)$, as much as possible, cf. \S\ref{sec:6}.

For these endeavours it is useful to complete a quadratic from $q:V \to R$ to a quadratic pair $(q,b)$ with a suitable bilinear companion $b: V \times V \to R$ of $q$ as introduced in \cite{QF1} and recalled in \S\ref{sec:1}. Most often we need the property that the pair $(q,b)$ is balanced, i.e., $q(x) = 2 b(x,x)$ for any $x \in V$,  and so we study also lifts and pushdowns of balanced quadratic pairs.

Further we need an account of \textbf{scalar extension} of an $R$-module $V$ by a semiring homomorphism $\lm: R' \to R$. For rings $R, R'$ the scalar extension of $V$ is provided by the tensor product $V \otimes_R R'$. But, since  tensor products for modules over semirings are a delicate matter, we prefer another method to obtain scalar extensions of $R$-modules.
This method is not constructive, cf. \S\ref{sec:4}, but nevertheless  suffices to establish  scalar extensions of bilinear forms and quadratic pairs with good categorical properties.

\section{Balanced companions and expansions
   of quadratic forms}\label{sec:1}

In what follows, $R$ is a semiring and $V$ is an $R$-module.
We recall, elaborate, and complement basic definitions and facts from \cite{QF1}
about quadratic forms on $V$, mainly \S1
therein.

A map $q: V\to R$ is a \bfem{quadratic form} if
\begin{equation}\label{eq:1.1}
q(cx)= c^2q(x)
\end{equation}
for all $x\in V,$ $c\in R,$ and there exists a symmetric
$R$-bilinear form $b: V\times V\to R,$ called a \bfem{companion}
of $q$, such that
\begin{equation}\label{eq:1.2}
q(x+y)=q(x)+q(y)+b(x,y)
\end{equation}
for all $x,y\in V.$ We then say that $(q,b)$ is a
\bfem{quadratic pair} on $V.$ It follows from \eqref{eq:1.1} and~
\eqref{eq:1.2} that for all $x\in V$
$$4q(x)=2q(x)+b(x,x).$$
We call the companion $b$ of $q$ \bfem{balanced}, if
\begin{equation}\label{eq:1.3}
 b(x,x)=2q(x),
\end{equation}
and  say that the quadratic pair $(q,b)$ is balanced.

If $B:V\times V\to R$ is any $R$-bilinear form, then it is
immediate that
\begin{equation}\label{eq:1.4}
q(x):=B(x,x)\qquad (x\in V)
\end{equation}
is a quadratic form on $V,$ and the symmetrized bilinear form
$b:=B+B^\ont,$ i.e.,
\begin{equation}\label{eq:1.5}
b(x,y)=B(x,y) + B(y,x)
\end{equation}
is a balanced companion of $q$.
We then say that  $B$ is an
\bfem{expansion} of the  quadratic form~ $q$, or the
quadratic pair $(q,b)$, and use the (new)
notation
\begin{equation}\label{eq:1.6}
q=[B].
\end{equation}

Given a system of generators $(\veps_i \ds|i\in I)$ of the
$R$-module $V$, we underpin the above definitions by explicit
formulas. For notational convenience, we choose a total ordering
of the index set~ $I$.

Let $q: V\to R$ be a quadratic form on $V$ and let $b: V\times V\to R$ be
a companion of $q.$ We write
\begin{equation}\label{eq:1.7}
\al_i:=q(\veps_i),\qquad\bt_{ij}:=b(\veps_i,\veps_j),
\end{equation}
for $i,j\in I.$ For $x,y\in V$  we choose presentations
\begin{equation}\label{eq:1.8}
x=\sum_{i\in I}x_i\veps_i,\qquad y=\sum_{i\in I} y_i\veps_i,
\end{equation}
with $x_i,y_i\in R;$ $x_i=0,$ $y_i=0$ for   almost all $i\in I.$
Then, if $I$ is finite,
\begin{equation}\label{eq:1.9}
q(x)=\sum_{i\in I}\al_ix_i^2+\sum_{i<j}\bt_{ij}x_ix_j,
\end{equation}
  by an easy induction on the cardinality $|I|$ of $I$, which thus  holds  also for
infinite ~ $I$. Clearly
\begin{equation}\label{eq:1.10}
b(x,y)=\sum_{i,j\in I}\bt_{ij}x_iy_j=\sum_{i\in
I}\bt_{ii}x_iy_i+\sum_{i<j}\bt_{ij}(x_iy_j+x_jy_i).
\end{equation}

If $b$ is balanced, then $b(\veps_i,\veps_i)=2q(\veps_i),$ i.e.,
for all $i\in I,$
\begin{equation}\label{eq:1.11}
\bt_{ii}=2\al_i.
\end{equation}
Conversely, when \eqref{eq:1.11} holds, it follows from
\eqref{eq:1.9} and \eqref{eq:1.10} that $b(x,x)=2q(x)$ for all~
$x\in  V.$ \textit{So $b$ is balancd iff
\eqref{eq:1.11} holds.}

Assume further  that $B:V\times V\to R$ is an expansion of $(q,b)$,
(in particular $(q,b)$ is balanced), and let
\begin{equation}\label{eq:1.12}
\gm_{ij}=B(\veps_i,\veps_j)\qquad (i,j\in I).
\end{equation}
Then clearly, for all $i\in I,$
\begin{equation}\label{eq:1.13}
\gm_{ii}=\al_i,
\end{equation}
and for $i\ne j$
\begin{equation}\label{eq:1.14}
\gm_{ij}+\gm_{ji}=\bt_{ij},
\end{equation}
and, of course,
\begin{equation}\label{eq:1.15}
B(x,y)=\sum_{i,j\in I}\gm_{ij}x_iy_j.
\end{equation}

These formulas tell us very little about the \textit{existence} of
balanced companions and/or expansions of a given quadratic form,
since $(\veps_i \ds |i\in I)$ is only a system of generators of
$V,$ and we cannot extract $q,b,B$ from pregiven coefficients
$\al_i,$ $\bt_{ij},$ $\gm_{ij}. $ This changes if $V$ is a free
$R$-module.

\begin{prop}\label{prop:1.1}Assume that $V$ is a free $R$-module with a totally ordered basis $(\veps_i \ds |i\in I)$.   Given a quadratic form $q: V \to R $,  the extensions $B:V\times V\to R$ of $q$ are all obtained as follows. Choose a companion\footnote{See \cite[\S6]{QF1} for the possible choices.} $b$ of $q$,  write   $\al_i:=q(\veps_i)$, $\bt_{ij}:=b(\veps_i,\veps_j)$ for $(i,j\in I),$ and define
$$B(x,y):=\sum_{i,j\in I}\gm_{ij}x_ix_j$$
with coefficients $\gm_{ii}:=2\al_i$ and $\gm_{ij},$ $\gm_{ji}$
for $i<j$ chosen in  a way that $\gm_{ij}+\gm_{ji}=\bt_{ij}.$
The associated balanced companion $\tlb=B+B^\ont$ of $q$ is given
by the formula
\begin{equation}\label{eq:1.16}
\tlb(x,y)=\sum_{i\in
I}(2\al_i)x_iy_i+\sum_{i<j}\bt_{ij}(x_iy_j+x_jy_i),
\end{equation}
and these symmetric bilinear forms $\tlb(x,y)$ are all balanced
companions of $q.$
\end{prop}

\begin{proof}
A very easy verification based on the formulas
\eqref{eq:1.7}--\eqref{eq:1.15}.
\end{proof}

\begin{comment}\label{comment:1.2}
In practice one usually chooses $\gm_{ij}=\bt_{ij},$ $\gm_{ji}=0$
for $i<j,$ to obtain the unique ``triangular'' expansion
$B_\nabla$ of $(q,b),$ cf. \cite[\S1]{QF1},\footnote{Denoted in \cite{QF1}
by $\nabla_q.$}. But, we  want to be independent of the choice of
a total ordering of the basis $(\veps_i \ds |i\in I).$ We  use
an ordering as above only to ease notation.
\end{comment}

\begin{cor}\label{cor1.3}
As before assume that $V$ is free with basis $(\veps_i \ds |i\in
I).$ Assume also that $R$ is zero sum free. i.e., $R \sm \00$ is closed under addition.
Suppose  that $q:V\to R$ is a quadratic form and $b$ is a
companion of $q.$ The following are equivalent:
\begin{enumerate} \eroman \dispace
\item $q$ has a unique expansion; \item If $b$  is a
companion of $q$, then $b(\veps_i,\veps_j)=0$ for all $i\ne j$ in
$I.$
\end{enumerate}
When (i) and  (ii) hold, $q$ has a unique balanced companion $\tlb,$
given by
$$\tlb(\veps_i,\veps_j)=\begin{cases} 2q(\veps_i)&\quad\text{if $i=j$},\\
0&\quad \text{else}.\end{cases}$$
\end{cor}

\begin{proof}
This is a consequence of the description of all expansions of $q$
in \propref{prop:1.1}.
\end{proof}

In the case of a finite  ordered basis $(\veps_i \ds |i\in I)$, $I=\{1,2,\dots,n\},$ we  use a more
elaborated  notation as follows. We denote the quadratic form
\eqref{eq:1.9} by the triangular scheme
\begin{equation}\label{eq:1.18}
q=\begin{bmatrix}
\al_1 & \bt_{12} &\hdots & \bt_{1n}\\
&\al_2\\
& & \ddots & \vdots\\
& & & \al_n
\end{bmatrix}
\end{equation}
as in \cite[\S1]{QF1}, and identify any bilinear form
$C:V\times V\to R$ with the $(n\times n)$-matrix
$(C(\veps_i,\veps_j))_{1\le i,j\le n}.$ With this notation, the
companion \eqref{eq:1.10} of $q$ reads
$$b=(\bt_{ij})_{1\le i,j\le n},$$
and the expansion \eqref{eq:1.15} of $q$ reads
$$B=(\gm_{ij})_{1\le i,j\le n}.$$
We have $$\gm_{ii}=\al_i, \qquad \gm_{ij}+\gm_{ji}=\bt_{ij},$$
whence the quadratic form $q(x)=B(x;x)$ reads (cf. \eqref{eq:1.6})
\begin{equation}\label{eq:1.19}
q=\begin{bmatrix} \gm_{11} & \hdots & \gm_{1n}\\
\vdots & \ddots & \vdots \\
\gm_{n1} & \cdots & \gm_{nn}\end{bmatrix}=\begin{bmatrix}
\al_1 & \gm_{12}+\gm_{21}, &\hdots & \gm_{1n}+\gm_{n1}\\
& \al_2 &  & \gm_{2n}+\gm_{n2}\\
& & \ddots &  \vdots \\
& & & \al_n\end{bmatrix}.
\end{equation}
The triangular expansion of $q$ (cf. Comment \ref{comment:1.2})
reads
\begin{equation}\label{eq:1.20}
B_\nabla=\begin{pmatrix}
\al_1 & \bt_{12} & \hdots & \bt_{1n}\\
0 & \al_{2} & \hdots & \bt_{2n}\\
\vdots  &  \ddots &  \ddots & \vdots \\
0 & \hdots & 0  &\al_n\\
 \end{pmatrix} .
\end{equation}
When $V$ is free with an infinite basis $(\veps_i \ds |i\in I)$, we
 use an analogous notation employing $I\times I$-matrices.

We include  specific notation for \textbf{diagonal forms}, i.e.,
polynomials \eqref{eq:1.9} with coefficients $\bt_{ij} = 0$. Given
$\al, \dots, \al_n \in R$, we abbreviate
\begin{equation}\label{eq:1.21}
[\al_1, \dots , \al_n ]= \begin{bmatrix}
\al_1 & & \hdots  & 0 \\
 & \al_2 &  & \vdots \\
& & \ddots & \\
& & & \al_n\end{bmatrix},
\end{equation}

\begin{equation}\label{eq:1.22}
\langle  \al_1, \dots , \al_n \rangle = \begin{pmatrix}
\al_1 & & \hdots  & 0 \\
 & \al_2 &  & \vdots \\
 \vdots & & \ddots & \\
 0 &  \hdots & & \al_n\end{pmatrix}.
\end{equation}

\begin{remark} If $q = [\al_1, \dots , \al_n ]$, then
\begin{align*} q(x+y) & = \sum_i\al_i(x_i + y_i)^2  \\ & =
\sum_i\al_i x_i ^2 +\sum_i\al_i y_i ^2 + \sum_i 2\al_i x_i y_i,
\end{align*}
and we conclude from above that $\langle  2\al_1, \dots , 2\al_n
\rangle $ is a balanced companion of $q$.
\end{remark}

\begin{defn}\label{defn:1.4} Let $V$ be an $R$-module, not necessarily free.  We say that a
quadratic form $q:V\to R$ is \bfem{expansive}, if there exists
some expansion $B:V\times V\to R$ of $q$.  We say that ~$q$ is
\bfem{balanceable}, if $q$ admits at least one balanced companion.
In the same vein, a balanced quadratic pair $(q,b)$ on
$V$ is called \bfem{expansive}, if $(q,b)$ has at least one expansion.
\end{defn}

\begin{defn}[{\cite[Definition~2.1]{QF1}}]\label{defn:1.4.a} A quadratic form $q: V \to R$ is called \textbf{quasilinear}, if
$q(x+y) = q(x) + q(y)$ for all $x,y \in V.$ In other terms, $q$ is
quasilinear iff $b= 0$ is a companion of~ $q$.
\end{defn}

Note that, when $q$ is quasilinear, the pair $(q,0)$ is not
balanced, except in the trivial case $q=0$.

\begin{examp}\label{examp:1.4}
If $(a+b)^2 = a^2 + b^2$ for all $a,b \in R$, then every diagonal
form \eqref{eq:1.21} over $R$ is quasilinear. In
particular, this happens  if the semiring $R$ is supertropical.
\end{examp}

We proceed with some examples of balanceable and/or
expansive quadratic forms. Note that, if $V$ is free, then
by \propref{prop:1.1} every quadratic form $q$ on $V$ is expansive
and (hence) balanceable.
We start with ``rigid'' forms, defined as follows.

\begin{defn}[{\cite[Definition~3.1]{QF1}}]\label{defn:1.5}
A quadratic form $q:V\to R$ is \bfem{rigid}, if it has only one
companion.
\end{defn}

\begin{examp}\label{examp:1.6}
Assume that $R$ is a subsemiring of a ring $R'.$ Let $V$ be any
$R$-module and let $q:V\to R$ a be  quadratic form. Then $q$ has exactly
one companion, namely,
$$b(x,y)=q(x+y)-q(x)-q(y),$$
computed in $R'.$ Clearly, $b(x,x)=2q(x).$ Thus $q$ is rigid, and
the unique companion $b$ is balanced.
\end{examp}

\begin{examp}[{cf. \cite[Proposition~3.4]{QF1}}]\label{examp:1.7}
Let $V$ be any $R$-module and let $(\veps_i \ds|i\in I)$ be a system of
generators of $V.$ Assume that $q:V\to R$ is a quadratic form with
$q(\veps_i)=0$ for all $i\in I.$ If $b$ is a companion of $q,$
then $q(\veps_i+\veps_j)=b(\veps_i,\veps_j)$ for all $i,j\in I$
and $b(\veps_i,\veps_i)=4q(\veps_i)=0=2q(\veps_i)$. We conclude
that $q$ is rigid and the unique companion $b$ of $q$ is
balanced. 
\end{examp}

\begin{remark}\label{rem:1.8}
It has been proved in \cite{QF1} that, if $V$ is free with basis
$(\veps_i \ds |i\in I),$ and the form~ $[1]$  over $R$ is
quasilinear (i.e., $(\al + \bt)^2=\al^2+\bt^2$ for all $\al,\bt\in
R),$ and moreover $2$ is not a zero divisor in $R$ (i.e.,
$2\al=0\Rightarrow \al=0$ in $R),$ then a form $q:V\to R$ is rigid
\textit{iff} $q(\veps_i)=0$ for all $i\in I$ (cf.
\cite[Thm. 3.5]{QF1}). For example, these conditions on $R$ hold
if~$R$ is supertropical.
\end{remark}

\begin{examp}\label{examp:1.9} Assume that $V$ is a submodule of a free $R$-module $V'.$ Let $(q',b')$ be a quadratic pair on $V'.$ It is evident that the restriction $(q' |V,b'|V)$ is a quadratic pair on $V.$\ \footnote{As usual, we denote the restriction of a bilinear form $C:V'\times V'\to R$ to $V\times V$ by $C|V$ (instead of $C|V\times V).$} Moreover, if $(q',b')$ is balanced, then $(q'|V,b'|V)$ is balanced. Also, if $B':V'\times V'\to R$ is an expansion of $q',$ then $B'|V$ is an expansion of $q'|V.$ We conclude by \propref{prop:1.1} that every quadratic form $q:V\to R$ which can be extended to a quadratic form on $V'$ is expansive and (hence) balanceable.
\end{examp}

The search for balanced and for expansive quadratic forms seems to
bear a close relation to the problem of ``lifting'' quadratic
pairs.

\begin{defn}\label{defn:1.10}
Assume that $\lm: R'\to R$ is  a surjective semiring homomorphism
and $\vrp: V'\to V$ is a $\lm$-linear\footnote{That is,  $\vrp$ is
additive and $\vrp(a'x')=\lm(a')\vrp(x')$ for $a'\in R',$ $x'\in
V'$. In the literature often the term ``semilinear with respect to $\lm$''
is used.} surjective map from an $R'$-module $V'$ onto an
$R$-module $V.$ Let $(q,b)$ be a quadratic pair on $V$ and let
$(q',b')$ be a quadratic pair on $V'.$ We say that $(q',b')$ is
\bfem{a ``lifting'' of} $(q,b)$ (with respect to $(\lm,\vrp))$, if
 $$\lm\circ q'=q\circ \vrp,\quad \lm\circ b'=b\circ(\vrp\times \vrp).$$
If only $q$ and $q'$ are given, we say that $q'$ is \bfem{a
lifting of} $q$, if $\lm\circ q'=q\circ\vrp$. If only $q$ or
$(q,b)$ is given, we say that $q$ or $(q,b)$ \bfem{can be lifted
to} $V'$ \bfem{via} $(\lm,\vrp)$,  if there exists a form~ $q'$ or a pair
$(q',b')$  with these properties.
\end{defn}

\begin{prop}\label{prop:1.8}
Assume that $\lm: R' \twoheadrightarrow R$ is given as in
Definition \ref{defn:1.10}, and that the $R'$-module $V'$ is free.
\begin{enumerate} \eroman  \dispace
\item Every quadratic form $q: V\to R$ can be lifted to a
quadratic form $q':V'\to R'.$

\item Every expansive quadratic
pair $(q,b)$ on $V$ can be lifted to an expansive quadratic pair
$(q',b')$ on $V'.$

\item Every balanced quadratic pair
$(q,b)$ on $V$ can be lifted to a balanced quadratic pair
$(q',b')$ on $V'.$
    \end{enumerate}
    \end{prop}

    \begin{proof}
    We choose a basis $(\veps_i' \ds|i\in I)$ of $V'$ and write  $\veps_i:=\vrp(\veps_i').$ Then
    $(\veps_i \ds|i\in I)$ is a system of generators of $V.$ We choose a total ordering on $I.$ \pSkip
    (i): Write $q$ as a polynomial \eqref{eq:1.9}, and choose preimages $\al_i',$ $\bt_{ij}'\in R$ of the coefficients $\al_i,$~ $\bt_{ij}$ under $\lm$. Then define $q'$ by the analogous polynomial, where the $\veps_i$ are replaced by the $\veps_i'$, and the coefficients $\al_i,\bt_{ij}$ are replaced by $\al_i',\bt'_{ij}.$ Clearly, $\lm\circ q'=q\circ \vrp.$ \pSkip
    (ii): Let $B: V\times V\to R$ be an expansion of $(q,b)$, and write  $\gm_{ij}:=B(\veps_i,\veps_j)$ for $i,j\in I.$  Choose preimages $\gm_{ij}'$ of the $\gm_{ij}$ under $\lm$ and define a bilinear from $B': V'\times V'\to R'$ by the rule $B'(\veps_i',\veps_j')=\gm_{ij}'.$ Clearly $\lm\circ B'=B\circ(\vrp\times\vrp).$ Then define $(q',b')$ as the quadratic pair with expansion $B',$ i.e., $b'=B'+(B')^\ont,$ $q'(x')=B'(x',x')$ for $x'\in V'.$ Clearly $(q',b')$ is a lifting of $(q,b).$ \pSkip
    (iii): Write $\al_i:=q(\veps_i),$ $\bt_{ij}:=b(\veps_i,\veps_j)$, for $i,j\in I$, to   have $\bt_{ii}=2\al_i$ (cf. \eqref{eq:1.11}). Choose for every $i\in I$ a preimage $\al_i'$ of $\al_i$ in $R'$, and for every $i<j$ a preimage $\bt_{ij}'$ of $\bt_{ij}$ in~$R'$. Write
    $$\bt_{ii}':=2\al_i',\qquad \bt_{ji}':=\bt_{ij}'.$$
    Then define $q'$ and $b'$ by the analogues of the polynomials \eqref{eq:1.9} and  \eqref{eq:1.10} with coefficients $\al_i',$ $\bt_{ij}'.$ Clearly $(q',b')$ is a balanced lifting of $(q,b).$
    \end{proof}

    Note that, we did not claim that every quadratic pair $(q,b)$ can be lifted to the free module~ $V'.$ In fact, this might  be false, as the following proposition reveals.

    \begin{prop}\label{prop:1.9}
    Assume that $R'$ is a subsemiring of a ring $R''$ and $\lm:R'\to R$ is a surjective homomorphism from $R'$ to a semiring $R.$ Assume further  that $\vrp:V'\to  V$ is a surjective $\lm$-linear map from an $R'$-module $V'$ to an $R$-module $V.$ If $(q,b)$ is a quadratic pair on $V,$ which can be lifted to $V'$ via $(\lm,\vrp)$, then the companion $b$ of $q$ must be balanced.
    \end{prop}

    \begin{proof} Let $(q',b')$ be a lifting of $(q,b)$. We know by Example \ref{examp:1.6} that $b'$ is a balanced companion of $q'$ (and is the unique companion of $q').$ For every $x'\in V'$ we  conclude from $b'(x',x')=2q'(x)$ that
    $$b(\vrp(x'),\vrp(x'))=\lm(b'(x',x'))=\lm(2q'(x'))=2q(\vrp(x')).$$
    Since $\vrp$ is surjective, this implies  that $b(x,x)=2q(x)$ for every $x\in V.$
    \end{proof}

    \section{Isometric maps}\label{sec:2}

    Let $R$ be any semiring, and assume that $V=(V,q)$ and  $V'=(V',q')$ are quadratic $R$-modules.

    \begin{defn}\label{defn:2.1}
    A map $\sig:V\to V'$ is called \bfem{isometric}, if $\sig$ is $R$-linear and $q'(\sig x)=q(x)$ for every $x\in V.$
    \end{defn}

    We first state some very basic facts about isometric maps.

    \begin{prop}\label{prop:2.2}
    Assume that $(\veps_i \ds |i\in I)$ is a system of generators of the $R$-module $V.$ Assume that $\sig: V\to V'$ is any $R$-linear map, where for any two different indices $i,j\in I$ the restriction $\sig|R\veps_i+R\veps_j$ is isometric. Then $\sig$ is isometric.
    \end{prop}

    \begin{proof}
   a) We first assume that $I=\{1,2,\dots,n\}$  is finite. Let $\veps_i':=\sig(\veps_i),$ and write
     $$\al_i:=q(\veps_i)=q'(\veps_i').$$
    Choose companions $b$ of $q$ and $b'$ of $q',$ and  for $i<j$ write
    $$\bt_{ij}:=b(\veps_i,\veps_j),\qquad \bt_{ij}':=b'(\veps_i',\veps_j').$$
    Given  $x\in V$, we choose a presentation
    $$x=\sum_{i=1}^nx_i\veps_i,$$
    and then have
    $$\sig(x)=\sum_{i=1}^nx_i\veps_i'.$$
    It follows (cf. \eqref{eq:1.9}) that
    \begin{equation}\label{eq:2.1}
    q(x)=\sum_{i=1}^n\al_ix_i^2+\sum_{i<j}\bt_{ij}x_ix_j,
    \end{equation}
    and
    \begin{equation}\label{eq:2.2}
    q'(\sig(x))=\sum_{i=1}^n\al_ix_i^2+\sum_{i<j}\bt_{ij}'x_ix_j.
    \end{equation}

    For any $i<j$, we have
  \begin{align*} \al_ix_i^2  +\al_jx_j^2+\bt_{ij}x_ix_j & =q(x_i\veps_i+x_j\veps_j)
\\&=q'(x_i\veps_i'+x_j\veps_j') \\& =\al_ix_i^2+\al_jx_j^2+\bt_{ij}'x_ix_j.
\end{align*}
Thus, the right hand side of \eqref{eq:2.1} remains unchanged
if $\bt_{ij}$ are replaced everywhere by $\bt_{ij}'$.   By ~\eqref{eq:2.2}
we obtain $q(x)=q'(\sig(x)).$ \pSkip
b) When  $I$ is infinite, we know that
$\sig\Big|\sum\limits_{i\in J}R\veps_i$ is isometric for every finite
$J\subset I.$ It follows trivially that $\sig$ is isometric.
\end{proof}

Taking $V'=V,$ $\sig=id_V,$ we obtain

\begin{cor}\label{cor:2.3}
Let $(\veps_i \ds|i\in I)$ be a system of generators of $V$ and
let $q:V\to R,$ $q':V\to R$ be quadratic forms on $V.$ Suppose
that, for any two different indices $i,j\in I$,
$$q|R\veps_i+R\veps_j=q'|R\veps_i+R\veps_j.$$ Then $q=q'.$
\end{cor}

\begin{prop}\label{prop:2.4}
Assume that $V$ is an $R$-module, and $(\veps_i \ds|i\in I)$ is a
system of generators of~$V,$ with $|I|>1.$ Let $q:V\to R$ be a
quadratic form, and let $b:V\times V\to R$ be a symmetric bilinear form.

\begin{enumerate} \eroman \dispace
\item Suppose that, for any two different indices $i,j\in I,$
$b|R\veps_i+R\veps_j$ is a companion of $q.$ Then $b$ is a
companion of $q.$

    \item If, in addition, $b(\veps_i,\veps_i)=2q(\veps_i)$ for every $i\in I,$ then $b$ is a balanced companion.
        \end{enumerate}
        \end{prop}

        \begin{proof}
        It suffices to prove the claim for $I$ finite, $I=\{1,\dots, n\}.$
        \pSkip
        (i): We choose a companion $b_0$ of $q$ and write, for any $i,j\in I$,
        \begin{align*} \al_i:=q(\veps_i), \qquad \bt_{ij}^0=b_0(\veps_i,\veps_j), \qquad
        \bt_{ij}:=b(\veps_i,\veps_j).
        \end{align*}
        Given $x,y\in V,$ we choose presentations
        $$x=\sum_{i=1}^nx_i\veps_i,\qquad y=\sum_{i=1}^ny_i\veps_i.$$
        Then (cf. \eqref{eq:1.9})
        \begin{equation*}\label{eq:2.3}
        q(x)=\sum_{i=1}^n\al_ix_i^2+\sum_{i<j}2\bt_{ij}^0x_ix_j,
        \end{equation*}
               \begin{equation*}\label{eq:2.4}
        q(y)=\sum_{i=1}^n\al_iy_i^2+\sum_{i<j}2\bt_{ij}^0y_iy_j.
        \end{equation*}
        Given $i<j,$ we have
        \begin{align}
        \al_ix_i^2+\al_jx_j^2+2\bt_{i,j}^0x_ix_j & =q(x_i\veps_i+x_j\veps_j) \nonumber\\ & =\al_ix_i^2
        +\al_jx_j^2+2\bt_{ij}x_ix_j\label{eq:2.5}
      \end{align}
        and
        \begin{align}
        \al_iy_i^2+\al_j^2y_j^2+2\bt_{ij}^0y_iy_j & =q(y_i\veps_i+y_i\veps_j) \nonumber \\ &=\al_i
       y_i^2+\al_jy_j^2+2\bt_{ij}y_iy_j,  \label{eq:2.6}
        \end{align}
        \begin{align} &
        q(x_i\veps_i+x_j\veps_j)+q(y_i\veps_i+y_j\veps_j)+\sum_{i,j=1}^n
        \bt_{ij}^0x_iy_j   \nonumber  \\ & =  q(x_i\veps_i+x_j\veps_j)+q(y_i\veps_i+y_j\veps_j)
        +\sum_{i,j=1}^n\bt_{ij}x_iy_j. \label{eq:2.7.0}
        \end{align}
        Then
        \begin{align*} q(x+y) & =q(x)+q(y)+b_0(x,y)
        \\ &=\sum_{i=1}^n\al_ix_i^2+\sum_{i<j}2\bt_{ij}^0x_ix_j+\sum_{i=1}^n\al_i
        y_i^2+\sum_{i<j}2\bt_{ij}^0y_i y_j+\sum_{i,j=1}^n\bt_{ij}^0x_iy_j.
        \end{align*}
        This sum remains unchanged  if we replace everywhere $\bt_{ij}^0$ by $\bt_{ij} ,$ due to
        \eqref{eq:2.5},  \eqref{eq:2.6},  and ~ \eqref{eq:2.7.0}.
        This tells us that
        $$q(x+y)=q(x)+q(y)+b(x,y).$$
\pSkip
        (ii): By \S\ref{sec:1} (cf. \eqref{eq:1.11}), the companion $b_0$ is balanced
        \begin{align*}
        &\Leftrightarrow  \bt_{ii}^0=2\al_i\ \forall i\\[1mm]
        &\Leftrightarrow b_0|R\veps_i+R\veps_j\ \text{balanced}\ \forall i,j\\[1mm]
        &\Leftrightarrow b |R\veps_i+R\veps_j\ \text{balanced}\ \forall i,j\\[1mm]
        &\Leftrightarrow \bt_{ii}=2\al_i \ \forall i\in I\\[1mm]
        &\Leftrightarrow b  \ \text{is balanced}.
        \end{align*}
        Since $B|R\veps_i$ is an expansion of $q|R\veps_i$, we  have
        \begin{equation}\label{eq:2.7}
        \gm_{ii}=\al_i
        \end{equation}
        for all $i\in I.$

        For any $i<j,$ we have
        $$q(x_i\veps_i+x_j\veps_j)=B(x_i\veps_i+x_j\veps_j,x_i\veps_i+x_j
        \veps_j),$$
        whence, by \eqref{eq:2.7},
        \begin{equation}\label{eq:2.8}
        \al_ix_i^2+\al_jx_j^2+\bt_{ij}^0x_ix_j=\gm_{ii}x_i^2+\gm_{jj}x_j^2+(\gm_{ij}+
        \gm_{ji})x_ix_j,
        \end{equation}
        and then, by \cite[Eq. (1.12)]{QF1},
        \begin{align*}q(x)&=\sum_{i=1}^n\al_ix_i^2+\sum_{i<j}\bt_{ij}^0x_ix_j\\ & \underset{\eqref{eq:2.7}\eqref{eq:2.8}}=
        \sum_{i=1}^n\gm_{ii}x_i^2+\sum_{i<j}(\gm_{ij}+\gm_{ji})x_ix_j\\
        &=B\bigg(\sum_{i=1}^nx_i\veps,
        \sum_{i=1}^nx_i\veps_i\bigg) \quad =B(x,x).
        \end{align*}
         \end{proof}

         \begin{prop}\label{prop:2.5} Assume  that $V$ is an $R$-module and $(\veps_i \ds|i\in I)$ is a system of generators of $V$, with $|I| >1$. Let $q : V \to R $ be a quadratic form and let $B :V\times V\to R$ be a  bilinear form (not necessarily symmetric).  Suppose that, for any two different $i,j \in I $,  $B|R\veps_i + R\veps_j$ is an expansion of $q|R\veps_i + R\veps_j$. \footnote{Recall that, if $W$ is a submodule of $V$, we write $B|W := B|W \times W$.} Then $B$ is an expansion of $q$.

         \end{prop}
        \begin{proof} We choose a companion $b_0$ of $q$. We may assume that $I = \{ 1, \dots, n\}$
 is finite. For any $i,j \in I$  we write
 \begin{align*} \al_i  := q(\veps_i), \qquad   \bt^0_{ij} := b_0(\veps_i, \veps_j), \qquad
  \gm_{ij} := B(\veps_i, \veps_j).
 \end{align*}
 Since $B| R\veps_i$ is an expansion of $q|R\veps_i$, we have

  \begin{equation}\label{prop:eq:2.5.a}
  \gm_{ii} = \al_i \tag{$*$}
 \end{equation}
 for all $i \in I$.
    For any $i < j $ we have
    $$ q(x_i\veps_i + x_j \veps_j) = B(x_i\veps_i + x_j \veps_j,x_i\veps_i + x_j \veps_j),$$
    whence, by $\eqref{prop:eq:2.5.a}$,

  \begin{equation}\label{prop:eq:2.5.b}
  \al_i x_i^2 +\al_jx_j^2  + \bt^0_{ij} x_i x_j =    \gm_{ii} x_i^2 +\gm_{jj} x_j^2  + (\gm_{ij} + \gm_{ji}) x_i x_j, \tag{$**$}
  \end{equation}
  and then, by \eqref{eq:1.9},
 \begin{align*}  q(x) & =
 \sum_1^n \al_i x_i^2 + \sum_{i<j} \bt^0_{ij} x_i x_j  \\
 & \underset{(*), (**)}{=}  \sum_1^n \gm_{ii} x_i^2   + \sum_{i<j} (\gm_{ij} + \gm_{ji}) x_i x_j \\ & =
 B\bigg( \sum_1^n  x_i \veps_i \; , \sum_1^n  x_i \veps_i \bigg) \quad  = B(x,x).
  \end{align*}
        \end{proof}

        For technical reasons, we also define isometric maps with respect to quadratic pairs.

        \begin{defn}\label{defn:2.6}
        Let $\sig: V\to V'$ be an $R$-linear map, and let $(q,b)$ and $(q',b')$ be quadratic pairs  on~ $V$ and  on $V'$, respectively. We say that $\sig$ is \bfem{isometric with respect to} $(q,b)$ \textbf{and} $(q',b'),$ if  $q'(\sig x)=q(x)$ for all $x\in V$ (i.e., $\sig$ is isometric with respect to $q$ and $q'$, as defined above) and $b'(\sig x,\sig y)=b(x,y)$ for all $x,y\in V.$ In short, $\sig$ is isometric with respect to $(q,b)$ and $(q',b')$ iff
        \begin{equation}\label{eq:2.9}
        q'=q\circ\sig,\qquad b'=b\circ(\sig\times \sig).
        \end{equation}
        \end{defn}

        We state a counterpart to \propref{prop:2.2}.

        \begin{prop}\label{prop:2.6}
        Assume that $\sig:V\to V'$ is an $R$-linear map, and that $(q,b)$ and $(q',b')$ are quadratic pairs on $V$ and $V'$ respectively. Assume further  that $(\veps_i\ds|i\in I)$ is a system of generators of the $R$-module $V,$ $|I|>1,$ and that $\sig|R\veps_i+R\veps_j$ is isometric with respect to $(q|R\veps_i+R\veps_j\, ,b|R\veps_i+\veps_j)$ and $(q',b')$ for all pairs of indices $i,j\in I.$ Then $\sig$ is isometric with respect to $(q,b)$ and $(q',b').$
        \end{prop}

        \begin{proof}
        In contrast to \propref{prop:2.2} this is a triviality. Choose a total ordering on $I$,  and
        let $\veps_i'=\sig(\veps_i).$ Then, for any $i,j\in I$, we have $b'(\veps_i',\veps_j')=b(\veps_i,\veps_j)$, which implies $b'\circ(\sig\times\sig)=b.$ Furthermore, $q'(\veps_i')=q(\veps_i)$ for all $i\in I.$ For $x=\sum\limits_i x_i\veps_i\in V,$ we have $\sig(x)=\sum\limits_i x_i\veps_i',$ and
        \begin{align*} q(x) & =\sum_i x_i^2q(\veps_i)+\sum_{i<j}x_ix_jb(\veps_i,\veps_j) \\ & =\sum_i x_i^2q'(\veps_i')+\sum_{i<j}x_ix_jb'(\veps_i',\veps_j')\quad  =q'(\sig(x)).\end{align*} \vskip -5mm
        \end{proof}

        \section{A strategy to describe subforms of a quadratic form on a free module}\label{sec:3}

        Assume that $(q,b)$ is a quadratic pair on a free $R$-module $V.$ Fixing a basis $(\veps_i\ds|i\in I)$ of~ $V,$  $q$ and $b$ can be described by matrices, as indicated in \S\ref{sec:1}. These matrices provide a useful computational approach to analyzing $q$ and $b.$

        As mentioned,  in the important case that $R\setminus\{0\}$ is closed under addition and multiplication, and also when $R$ is any supertropical semiring, the basis $(\veps_i \ds|i\in I)$ of $V$ is unique up to multiplication of the $\veps_i$ by units of $R$ (cf. \cite[Therorem 0.9]{IKR1}). Therefore, $(q,b)$ is a very rigid object, well amenable to combinatorial arguments.
        But even then, in general, an easy description of ``subforms'' of $q$ and ``subpairs'' of $(q,b)$, defined below, is not always accessible.

        \begin{defn}\label{defn3.1}
        A \bfem{subform} of a quadratic form $q:V\to R$ is a restriction $q| U:U \to R$ of~ $q$ to an $R$-submodule $U$ of $V.$ A \bfem{subpair} of a quadratic pair $(q,b)$ on $V$ is a pair
        $(q|U,b|U)$ on a submodule $U$ of $V.$
        \end{defn}

        Submodules of free modules most often are not free. Nevertheless, as 
        pointed in  \S\ref{sec:1} (Example \ref{examp:1.9}),  if $V$ is free, then a subform of $q:V\to R$  is still a decent object: it is expansive and (hence) balanceable.

        We next  outline a strategy to understand subforms and subpairs. We choose a system of generators $(\eta_k \ds|k\in K)$ of $U$, take a free $R$-module $U'$ with basis $(\eta_k'\ds|k
       \in K)$ (same index set~$K),$ and define an $R$-linear map $\vrp: U'\to V$ by letting
       $$\vrp\Big(\sum_{k\in K}x_k\eta_k'\Big):=\sum_{k\in K}x_k\eta_k$$
       ($x_k\in R,$ almost all $x_k=0).$ From $q$ and $(q,b)$, we obtain a quadratic form $q':=q\circ\vrp$ and a quadratic pair $(q',b'):=(q\circ\vrp,b\circ(\vrp\times\vrp))$ on $U'$ which can be described by matrices. Then we utilize  the fact that the map $\vrp:U'\to V$  is isometric with respect to $(q',b')$ and $(q,b)$ by construction.

       \begin{remark}\label{rem:3.2}
In this regard, we mention the intriguing fact that, in the
case that $R$ is a supersemifield and $U$ is a submodule of the
free $R$-module $V,$ any set of the so-called \bfem{critical
vectors} of $U$ (cf. \cite[Definition 5.17]{IKR-LinAlg}) is the
\bfem{unique minimal} set of generators of the module
$U_1=\sum\limits_{s\in S} Rs,$ up to multiplication of the
elements of $S$ by units of $R,$ \footnote{Of course, assuming
$R^*s\ne R^*t$ for any different $s,t\in S.$} cf. \cite[~Corollary~5.25,
Proposition~5.27]{IKR-LinAlg}.  \{Actually, the result in
\cite[\S5]{IKR-LinAlg} is stronger. It works with a notion of
``tropically spanning sets'', tailored to the needs of
supertropical linear algebra, instead of our sets of generators,
there called ``classically spanning sets''.\} So $V$ contains
many submodules with a canonical set of generators, although these
submodules most often are not free modules.
\end{remark}

To gain enough flexibility, in order to cope with subforms and subpairs,
we introduce more terminology.

\begin{defn}\label{defn:3.3}
Assume that $(q',b')$ and $(q,b)$ are quadratic pairs on
$R$-modules $U'$ and $V,$ respectively. We say that $(q',b')$ is
\bfem{subordinate} to $(q,b)$, and write $(q',b')\prec(q,b),$ if
there exists an $R$-linear map $\vrp :U'\to V$ such that
$$q'=q\circ\vrp,\qquad b'=b\circ(\vrp\times \vrp).$$
We then also say that $(q',b')$ is \bfem{subordinate  to} $(q,b)$
\bfem{by the map} $\vrp$ (or, \textbf{via} $\vrp$), and write
$(q',b')\prec_\vrp(q,b).$ In a similar vein, we say that $q'$ is
\bfem{subordinate to} $q$ (via $\vrp)$, if $q'=q\circ\vrp,$ and
then write $q'\prec q$ or $q'\prec_\vrp q.$
\end{defn}

\begin{comment*}
Clearly, $q'\prec_\vrp q$ (resp. $(q',b')\prec_q(q,b)$)  iff the
map $\vrp: U'\to V$ is isometric with respect to $q'$ and $q$
(resp. $(q',b')$ and $(q,b)$), as defined in \S\ref{sec:2}
(Definitions \ref{defn:2.1} and \ref{defn:2.6}). So the new terms
describe a situation that already appeared in \S\ref{sec:2}, but
they suggest a different perspective than taken in \S\ref{sec:2}.
\end{comment*}

If $\vrp$ is injective, then $q'\prec_\vrp q$ means that $\vrp$ is
an embedding of the quadratic module $(U',q')$ into $(V,q).$
Subforms are special cases of subordinate forms.

\begin{remarks}\label{rem3.4}
Let $\vrp: U'\to V$ be an $R$-linear map.
\begin{enumerate} \eroman \dispace
\item If $q'\prec_\vrp q$ and $q$ is quasilinear, then $q'$ is
quasilinear. Indeed, for $x',y'\in U'$ we have
     \begin{align*} q'(x'+y') & =q(\vrp(x'+y')) \\ & =q(\vrp(x'))+q(\vrp(y')) \\ & =q'(x')+q'(y'). \end{align*}
    \item  If $(q',b')\prec_\vrp(q,b)$ and $(q,b)$ is balanced, then $(q',b')$ is balanced.
    \item If $(q,b)$ has an expansion $B:V\times V\to R,$ then $(q',b')$ has the expansion $B\circ(\vrp\times\vrp).$
        \end{enumerate}
        \end{remarks}

        \begin{schol}\label{schol:3.5}
        Assume that $V$ and $U'$ are free with finite bases $(\veps_i,\, 1\le i \le n)$ and  $(\eta_k,\, 1\le k\le m)$, respectively.  Assume further that the quadratic pair $(q,b)$ on $V$ is balanced. Then a description of $\vrp,$ $(q,b),$ and $(q',b')$ works as follows. We identify a  vector $x=\sum\limits_i^nx_i\veps_i$ in~ $V$ with the column $\begin{pmatrix} x_1\\ \vdots\\ x_n\end{pmatrix}$ and a vector $x'=\sum\limits_k^mx_k'\eta_k$ in $U'$ with the column
$\begin{pmatrix} x_1'\\ \vdots\\ x_m'\end{pmatrix}$. Thus
$\vrp(x)=Sx$ with $S$ an $m\times n$-matrix. We choose an
expansion $B:V\times V\to R$ of $(q,b)$ and identify it with the
$n\times n$-matrix
$$A=(B(\veps_i,\veps_j))_{1\le i,j\le n}.$$
\{Recall that $A$ can be chosen trigonal.\} Then
\begin{equation}\label{eq:3.1}
q(x)= \trn{x}Ax,\qquad
b(x,y)= \trn{x}\big(A+\trn{A}\big)y.
\end{equation}
$(q',b')$ has the expansion $B\circ(\vrp\times\vrp)$ with matrix
$\trn{S}AS,$ and so
\begin{equation}\label{eq:3.3}
q'(x')=\trn{x'}\trn{S}ASx', \qquad
b'(x',y')=\trn{x'}(\trn{S}AS+\trn{S}\trn{A}S)y'.
\end{equation}
\end{schol}

\newpage
\begin{examp}\label{examp:3.6} $ $
\begin{enumerate} \ealph \dispace
  \item
 Assume that $n=m=2,$
$$S=\begin{pmatrix}
 1 & d\\ c& 1\end{pmatrix},\qquad A=\begin{pmatrix} a_{11} & a_{12}\\
 0 & a_{22}\end{pmatrix},$$
 with $c,d\in R$, $a_{ij}\in R.$ We obtain
 $$\trn{S}AS=\begin{pmatrix}[ll] a_{11}+ca_{12}+c^2a_{22},&da_{11}+a_{12}+ca_{22}\ \ \\
 da_{11}+cda_{12}+ca_{22},&a_{22}+da_{12}+d^2a_{11}\end{pmatrix}.
$$
Thus, $q'=[\trn{S}AS]$ (cf. notation \eqref{eq:1.6}) has the
trigonal description
$$q'=\begin{bmatrix}[ll] a_{11}+ca_{12}+c^2a_{22},&a_{12}+2da_{11}+2ca_{22}+cda_{12}\\
& a_{22}+da_{12}+d^2a_{11} \end{bmatrix} . $$

\item In particular, if $a_{12} =0 $, then
$$q'=\begin{bmatrix}[ll] a_{11}+c^2a_{22},&2da_{11}+2ca_{22}\\
& a_{22}+d^2a_{11} \end{bmatrix} . $$ Note that if the form $[1]$
is quasilinear over $R$ (i.e., $(a + b)^2 = a^2 + b^2$), then $q$
is quasilinear, and hence $q'$ is quasilinear  by Remark
\ref{rem3.4}.(i).

\item On the other hand, if $a_{11} = a_{22} =0 $, whence $R$
is rigid (cf. Example \ref{examp:1.7}), and $c,d \neq 0$, then
$$q'=\begin{bmatrix}[ll] ca_{12},&(1+cd)a_{12}\\
& da_{12} \end{bmatrix}.  $$  So often, e.g., when $R$ is
supertropical, $q'$ is not rigid (cf. Remark \ref{rem:1.8}).
\end{enumerate}

\end{examp}

\section{Scalar extension of bilinear forms and of expansive quadratic pairs}\label{sec:4}

In  the sequel, the map  $\lm: R\to R'$ is a semiring homomorphism.
\begin{defn}\label{defn:4.1}
Let $V$ be an $R$-module. A $\lm$-\bfem{scalar extension} of $V$
is an $R'$-module $V'$ together with a $\lm$-linear map $\rho:
V\to V'$ which has the following universal property: For any
$\lm$-linear map $\vrp: V\to W',$ there exists a unique
$R'$-linear map $\psi: V'\to W'$ such that $\psi\circ\rho=\vrp.$
\end{defn}

It is evident that, if such a map $\rho: V\to V'$ exists, then it is
unique in a strong sense. Namely, for any other $\lm$-scalar
extension $\rho_1: V\to V_1'$ there exists a unique $R'$-linear
\textit{isomorphism} $\al: V'\isoto V_1'$ with
$\al\circ\rho=\rho_1.$

 If $R$ and $R'$ are rings, then it is common to obtain $\rho: V\to V'$ by a tensor product construction $V'=R'\otimes_{R,\lm}V,$ but for semirings tensor products can be an intricate matter in general (cf. \cite{QFSym}). So we take a different approach.

 \begin{prop}\label{prop:4.2}
 Assume that $V$ is free with basis $(\veps_i\ds|i\in I)$.  Let $V'$ be a free $R'$-module with basis $(\veps_i' \ds|i\in I)$ (same index set $I).$ Then the $\lm$-semilinear map $\rho: V\to V'$ given by
 $$\rho\Big(\sum_{i\in I} x_i\veps_i\Big):=\sum_{i\in I}\lm(x_i)\veps_i'$$
is a $\lm$-scalar extension of $V.$
\end{prop}

\begin{proof}
Given a $\lm$-linear map $\vrp: V\to W'$, we define the map $\psi: V'\to
W'$ by
$$\psi\Big(\sum_{i\in I} x_i'\veps_i'\Big):=\sum_{i\in I} x_i'\vrp(\veps_i).$$
Clearly, $\psi$ is $R'$-linear and $\psi\circ\rho=\vrp.$ Since
$\rho(V)$ generates the $R'$-module $V',$ $\psi$ is the unique
such map.
\end{proof}

We start proving (in a nonconstructive way) that a
$\lm$-scalar extension $\rho: V\to V'$ exists for every $R$-module
$V.$

\begin{defn}\label{defn:4.3}
An equivalence relation $E$ on an $R$-module
$V$ is called $R$-\bfem{linear}, if for  $x_1,x_2\in V$ the following hold:
\begin{enumerate} \eroman \dispace
\item $x_1\sim_Ex_2\Rightarrow x_1+y\sim_Ex_2+y$ for any $y\in
V$ (whence $x_1\sim_E x_2,$ $y_1\sim_E y_2\Rightarrow
x_1+y_1\sim_Ex_2+y_2),$
    \item $x_1\sim_Ex_2,c\in R\Rightarrow cx_1\sim_Ecx_2.$
    \end{enumerate}
    \end{defn}
  \begin{remark}\label{rem:4.4}
    If $E$ is $R$-linear, then we can equip the set $V/E$ of $E$-equivalence classes with the structure of an $R$-module such that the following holds:
    $$[x]_E+[y]_E:=[x+y]_E,\qquad c[x]_E=[cx]_E.$$
    This is the unique $R$-module structure on $V/E$ for which  the map $$\pi_E:V\To V/E$$ is $R$-linear.
    \end{remark}

    \begin{thm}\label{thm:4.5}
    Let $\lm: R\to R'$ be a semiring homomorphism and $V$ be an $R$-module. Then there exists a $\lm$-scalar extension $\rho: V\to V'.$ The $R'$-module $V'$ is generated by the set $\rho(V).$
    \end{thm}

    \begin{proof}
    We choose a free $R$-module $\tlV$ together with a surjective $R$-linear map $\pi: \tlV\twoheadrightarrow V.$ By \propref{prop:4.2}, there exists a $\lm$-scalar extension $\tlrho: \tlV\to\tlV'$ (and $\tlV'$ is a free $R'$-module). Let $E(\pi)$ denote the equivalence relation on $\tlV$ determined by the map $\pi$. It consists of all pairs $(x_1,x_2)\in\tlV\times \tlV$ with $\pi(x_1)=\pi(x_2).$ \footnote{As usual, we view any binary relation on a set $S$ as a subset of $S\times S.$} $E(\pi)$ provides us with the subset $(\tlrho\times \tlrho)(E(\pi))$ of $\tlV'\times\tlV'.$ Let $\mathcal M$ denote the set of all $R'$-linear equivalence relations on $\tlV'$ containing this subset of $\tlV'\times \tlV'.$  $\mathcal M$ is not empty, since the set $\tlV'\times \tlV'$ is a member of $\mathcal M$ (all elements of $\tlV'$ equivalent). Let $E$ denote the intersection $\bigcap \mathcal M.$ It is the finest $R'$-linear  equivalence relation on $\tlV'$ such that
    $$\pi(x_1)=\pi(x_2)\dss \Rightarrow \tlrho(x_1)\sim\tlrho(x_2)$$
    for all $x_1,x_2\in\tlV'.$ The map $\pi_E\circ\tlrho:\tlV'\to\tlV'/E$ induces a well-defined map
    $$\rho: V\To\tlV'/E$$
    with
    \begin{equation}\label{eq:4.1}
    \rho(\pi(x))=[\tlrho(x)]_E=\pi_E(\tlrho(x))
    \end{equation}
    for all $x\in V.$ \footnote{See diagram \eqref{eq:4.2} below.}

    We claim that $\rho$ is a $\lm$-scalar extension of $V.$ Let $\vrp:V\to W'$ be any $\lm$-linear map of $V.$ Then $\vrp\circ\pi:\tlV\to W'$ is again $\lm$-linear. Since $\tlrho$ is a $\lm$-scalar extension of~ $\tlV,$ there exists a unique $R'$-linear map $\tlpsi:\tlV'\to W'$ with $\tlpsi\circ\tlrho=\vrp\circ\pi.$ If $x_1,x_2\in\tlV$ and $\pi(x_1)=\pi(x_2),$ then $$\tlpsi\tlrho(x_1)=\vrp\pi(x_1)=\vrp\pi(x_2)=\tlpsi\tlrho(x_2),$$ whence $(\tlrho(x_1),\tlrho(x_2))\in E(\tlpsi).$ Thus $E(\tlpsi)\in\mathcal M.$ We conclude that $E(\tlpsi)\supset E,$ which implies that there is a unique map $\psi:\tlV'/E\to V'$ with $\psi\circ\pi_E=\tlpsi.$ Therefore, by \eqref{eq:4.1},
    $$\vrp\circ\pi=\tlpsi\circ\tlrho=\psi\circ\pi_E\circ\tlrho=\psi\circ\rho\circ\pi.$$
    Since $\pi$ is surjective, it follows that $\vrp=\psi\circ\rho.$ We have obtained a commuting diagram
    \begin{equation}\label{eq:4.2}
    \begin{gathered}
    \xymatrix{
    \tlV \ar@{>>}[d]^\pi
    \ar[r]^{\tlrho}
    &   \tlV'   \ar@{>>}[d]^{\pi_E}  \ar@/^1pc/[dd]<2ex> ^{\tlpsi}
     \\
        V\ar[r]^\rho\ar[dr]_\vrp & \tlV'/E  \ar@{-->}^{\psi}[d]   \\
      &  W'
    }
    \end{gathered}
    \end{equation}

     The set $\rho(V)$ generates the $R'$-module $\tlV'/E$ since apparently $\tlrho(\tlV)$ generates the $R'$-module $\tlV'.$ Thus $\psi$ is unique with $\psi\circ\rho=\vrp.$  This completes the proof that $\rho: V\to\tlV'/E$ is a $\lm$-scalar extension of $V.$
    \end{proof}

    \begin{notation}\label{notation:4.6}
    We fix one $\lm$-scalar extension $\rho: V\to V'$ and denote it by
    $$\rho_V^\lm: V\To V_\lm,$$
    often writing $\rho_V:V\to V_{R'}$ or $\rho:V\to V_{R'}$ instead, if it is clear from the context which  homomorphism $\lm : R \to R'$ is  under consideration.
     \end{notation}

    We state an immediate consequence of \thmref{thm:4.5}.

    \begin{cor}\label{cor:4.7}
    Given an $R$-linear map $f:V\to W$, there exists a unique $R'$-linear map $f': V_\lm\to W_\lm$ such that the diagram
    $$\xymatrix{
    V \ar[r]^f \ar[d]_{\rho_V^\lm} \ar[r]^f & W\ar[d]^{\rho_W^\lm}\\
    V_\lm\ar[r]_{f'}&W_\lm
    }
    $$
    commutes. If $f$ is a linear isomorphism, then so is $f'.$
     \end{cor}

\begin{defn}\label{defn:4.8}
We call $f':V_\lm\to W_\lm$ the $\lm$-\bfem{scalar extension} of
$f,$ and write $f'=f_{R'}$, or more precisely $f'=f_\lm.$
\end{defn}

\begin{remark}\label{rem:4.9}
If $V$ and $W$ are $R$-modules, then $\Hom_R(V,W)$ is an
$R$-module in the obvious way: For $f\in\Hom_R(V,W)$, $c\in R,$
$v\in V:$
$$(c \cdot f)(v):= c  (f(v))=f( c v).$$
\{We use the assumption that $R$ is commutative.\}\\
The scalar extension map $f\mapsto f_{R'}$ from $\Hom_R(V,W)$ to
$\Hom_{R'}(V_{R'},W_{R'})$ is $\lm$-linear, as is easily checked.
In particular (take  $W=R$), we have a $\lm$-linear map
$$\sig:\Hom_R(V,R)\To\Hom_{R'}(V_{R'},R'),$$
with $\sig(f):=f_{R'}$ for $f\in\Hom_R(V,R).$ Note that, with
$\rho:=\rho^\lm _V,$
\begin{equation}\label{eq:4.3}
\sig(f)(\rho(x))=\lm(f(x))
\end{equation}
for all $x\in V.$
\end{remark}

We next establish scalar extensions for bilinear forms on an
$R$-module $V.$ As before, we fix a semiring homomorphism $\lm:
R\to R',$ and then have the scalar extension
$$\rho:=\rho^\lm_V: V\To V_{R'}$$
of $V$.

\begin{thm}\label{thm:4.10}
Given an $R$-bilinear form $B:V\times V\to R$, there exists a
unique $R'$-bilinear form $B':V_{R'}\times V_{R'}\to R'$ with
\begin{equation}\label{eq:4.4.0}B'\circ(\rho\times\rho)=\lm\circ B,
\end{equation}
where $\rho=\rho^\lm_V.$
\end{thm}

\begin{proof}
The bilinear form $B:V\times V\to R$ corresponds uniquely to a
linear map
$$\vrp_B:V\To \Hom(V,R),$$
defined by
\begin{equation}\label{eq:4.4}
\vrp_B(y)(x)=B(x,y)
\end{equation}
for $x,y\in V.$ \{We have chosen the ``righthand'' convention\} We
 write  briefly $\vrp = \vrp_B$ and  $V'=V_{R'}.$ Composing
$\vrp$ with the $\lm$-linear map
$$\sig:\Hom_R(V,R)\To\Hom_{R'}(V',R')$$ (cf. Remark
\ref{rem:4.9}), we obtain a $\lm$-linear map
$$\sig\circ\vrp:V\To\Hom_{R'}(V',R').$$
Thus $\sig\circ\vrp=\psi\circ\rho$ with a unique $R'$-linear map
$$\psi: V'\To \Hom_{R'}(V',R').$$ The map $\psi$ corresponds to an
$R'$-bilinear form $B':V'\times V'\to R,$ with
$$B'(x',y')=\psi(y')(x')$$
for $x',y'\in V'.$ For  $x,y\in V,$ by use of
\eqref{eq:4.3}, there with $f=\vrp(y),$ and \eqref{eq:4.4},  we obtain  that
 \begin{align*} B'(\rho(x),\rho(y)) & =((\psi\circ\rho)(y))(\rho(x)) \\ & =((\sig\circ\vrp)(y))(\rho(x)) \\& =\lm(\vrp(y)(x)) \\& =\lm(B(x,y)), \end{align*}
as desired. Uniqueness of $B'$ is evident, since $\rho(V)$ generates $V'=V_{R'}$ (cf. \thmref{thm:4.5}).
\end{proof}

\begin{defn}\label{defn:4.11}
We call $B'$ the $\lm$-\bfem{scalar extension of the bilinear
form} $B,$ and write $B'=B_{R'}$ (or, when necessary, $B'=B_\lm).$
\end{defn}

\begin{example}\label{examp:4.12}
Assume that $V$ is a free $R$-module with basis $(\veps_i\ds|i\in
I)$. Then, as we know, $V_{R'}$ is a free $R'$-module with basis
$(\rho(\veps_i)\ds |i\in I).$ A bilinear form $B$ on $V$ is given
by an $I\times I$-matrix $(\gm_{ij}|(i,j)\in I\times I)$,
$\gm_{ij}=B(\veps_i,\veps_j).$ It follows that $B_{R'}$ is given
by the matrix $(\lm(\gm_{ij})|(i,j)\in I\times I).$
\end{example}

\begin{prop}[Functoriality of $B \rightsquigarrow B_{R'}$]\label{prop:4.13}
If $R$-bilinear forms $$B:V\times V\To R,
\qquad B_1:U\times U\To R,$$ and an $R$-linear map $\vrp: U\to V$
with $B_1=B\circ(\vrp\times\vrp)$ are given,  then
$$(B_1)_{R'}=B_{R'}\circ(\vrp_{R'}\times\vrp_{R'}).$$
\end{prop}

\begin{proof}
Let $\vrp':=\vrp_{R'},$ $\rho:=\rho^\lm_V,$ $\rho_1:=\rho_U^\lm,$
$B':=B_{R'}.$ Then $\vrp'\circ\rho_1=\rho\circ\vrp,$ and thus
\begin{align*}
B'\circ(\vrp'\times \vrp')\circ(\rho_1\times\rho_1)&=B'\circ(\vrp'\rho_1\times\vrp'\rho_1) \\ & =B'\circ(\rho\vrp\times\rho\vrp)\\
&\underset{\eqref{eq:4.4.0}}{=} B'\circ(\rho\times \rho)\circ(\vrp\times\vrp)\\
&  = \lm \circ B\circ(\vrp\times\vrp) \\ & = \lm \circ B_1.
\end{align*}
This proves that $(B_1)_{R'}=B'\circ(\vrp'\times\vrp'),$ as
desired.
\end{proof}

\begin{prop}\label{prop:4.14}
Assume that $B_1$ and $B_2$ are bilinear forms on an $R$-module
$V$ which expand the same quadratic pair $(q,b)$ on $V,$ i.e.,
$$q(x)=B_1(x,x)=B_2(x,x)$$
for all $x\in V$, and
$$b=B_1+(B_1)^\ont =B_2+(B_2)^\ont.$$
Then the $\lm$-scalar extensions $(B_1)_{R'}$ and $(B_2)_{R'}$
expand the quadratic pairs $(q_1',b_1'),$ $(q_2',b_2')$ on~$V_{R'}$, and these  are again equal.
\end{prop}

\begin{proof} We have
$$b_1'=(B_1)_{R'}+(B_1^\ont)_{R'}=(B_1+B_1^\ont)_{R'}=(B_2+B_2^\ont)_{R'}=b_2'.$$
Let $(\veps_i\ds|i\in I)$ be a system of generators of the
$R$-module $V$ and put $\veps_i':=\rho(\veps_i).$ Then
$(\veps_i'\ds|i\in I)$ is a system of generators of $V_{R'}.$ We
choose a total ordering on $I$. For any $i,j\in I$, we have
$b_1'(\veps_i',\veps_j')=b_2'(\veps_i',\veps_j')$ and
$$q_1'(\veps_i')=(B_1)_{R'}(\veps_i',\veps_i')=\lm(B_1(\veps_i,\veps_i))=
\lm(B_2(\veps_i,
\veps_i))=(B_2)_{R'}(\veps_i',\veps_i')=q_2'(\veps_i).$$ We
conclude that. for any $z=\sum\limits_{i\in I}z_i\veps_i'\in V'$
(all $z_i\in R',$ almost all $z_i=0)$,
\begin{align*}
q_1'(z)&=\sum_{i\in I}z_i^2q_1'(\veps_i)+\sum_{i<j} z_iz_jb_1'(\veps_i,\veps_j)\\
&=\sum_{i\in I} z_i^2q_2'(\veps_i)+\sum_{i<j} z_iz_jb_2'(\veps_i,\veps_j)\ =q_2'(z).
\end{align*}
Thus $q_1'=q_2'.$
\end{proof}

\begin{defn}\label{defn:4.15}
Assume that $(q,b)$ is an expansive quadratic pair on an
$R$-module $V.$ We choose an expansion $B:V\times V\to R$ of
$(q,b).$ Let $(q',b')$ be the quadratic pair on $V_{R'}$, expanded
by the $\lm$-scalar extension $B_{R'}$ on $V_{R'}.$ This pair does
not depend on the choice of~ $B,$ as just proved. We call $(q',b')$
the $\lm$-\bfem{scalar expansion of} $(q,b)$, and call $q'$ the
\textbf{$\lm$-scalar expansion of the expansive quadratic form}
$q.$ We write $q'=q_{R'}$, $b'=b_{R'},$ or, when necessary,
$q'=q_\lm,$ $b'=b_\lm.$
\end{defn}

\begin{schol}\label{scol:4.16}
Assume that $(\veps_i\ds|i\in I)$ is a system of generators of an
$R$-module $V.$ Let ~$(q,b)$ be an expansive quadratic pair on $V.$
Put
$$q(\veps_i)=\al_i,\qquad b(\veps_i,\veps_j)=\bt_{ij},$$
and $\veps_i':=\rho_V(\veps_i).$ Then $(\veps_i'\ds|i\in I)$ is a
system of generators of the $R'$-module $V_{R'},$ and
$$q_{R'}(\veps_i')=\lm(\al_i),\qquad b_{R'}(\veps_i',\veps_j')=\lm(\bt_{ij}).$$
\end{schol}

\begin{examp}\label{examp:4.17}
Assume that $\lm:R\to R'$ is surjective. Then, for any $R$-module
$V$,  the $\lm$-linear map $\rho_V^\lm:V\to V_{R'}$ is again
surjective, since the $R'$-module $V_{R'}$ is generated by
$\rho_V^\lm(V),$ and $R'=\lm(R).$ Furthermore, if $(q,b)$ is a
quadratic pair on $V_{R'},$ then $(q',b')$ is the $\lm$-scalar
extension $(q_{R'},b_{R'})$ of $(q,b)$ iff $(q,b)$ is a lifting of
$(q',b')$ with respect to $(\lm,\rho_V^\lm),$ cf. Definition \ref{defn:1.10}.
\end{examp}

\section{$\lm$-isometric maps}\label{sec:5}

In what follows $\lm:R\to R'$ is a homomorphism from a
semiring $R$ to a semiring $R'.$

\begin{defn}\label{defn:5.1} Assume that $V$ is an $R$-module and $V'$ is an $R'$-module. If $q: V\to R$ and $q':V'\to R'$ are quadratic forms on $V$ and $V',$ respectively, then a $\lm$-linear map $\vrp: V\to V'$ is called $\lm$-\bfem{isometric}, with respect to $q$ and $q'$,  if
\begin{equation}\label{eq:5.1}
q'(\vrp(x))=\lm(q(x))
\end{equation}
 for all $x\in V$.

Given companions $b$ and $b'$ of $q$ and $q'$, we say that $\vrp$
is \bfem{isometric with respect to the quadratic pairs} $(q,b)$
and $(q',b')$, if in addition
\begin{equation}\label{eq:5.2}
b'(\vrp x ,\vrp y)=\lm(b(x,y))
\end{equation}
for all $x,y\in V$.
\end{defn}

Note that in the case that $R'=R,$ $\lm=\id_R,$ ``$\lm$-isometric''
means ``isometric'', as defined in \S\ref{sec:2} (Definitions
\ref{defn:2.1} and \ref{defn:2.6}).

\begin{comment*}
We already studied surjective $\lm$-isometric maps in
\S\ref{sec:1}, cf. Definition \ref{defn:1.10}. If, in the notation
there, $(q',b')$ is a lifting of $(q,b)$ with respect to $\lm:R'\to
R$ and $\vrp:V'\to V,$ then ~$\vrp$ is $\lm$-isometric with
respect to $(q',b')$ and $(q,b)$. The maps $\vrp$ there are the
same objects as the present surjective $\lm$-isometric maps. Note that in the present setting the role of $R',V'$ and $R,V$ has
been interchanged.
We begun with a homomorphism $\lm: R\to R'$
instead of $\lm:R'\to R.$
\end{comment*}
\begin{examp}\label{examp:5.2}
If $(q,b)$ is an expansive quadratic pair on an $R$-module $V,$
then the $\lm$-linear map
$$\rho_V^\lm:V\To V_\lm=V_{R'}$$
introduced in \S\ref{sec:4} (Notation \ref{notation:4.6}) is
$\lm$-isometric with respect to $(q,b)$ and its scalar extension is
$(q_{R'},b_{R'}).$
\end{examp}

\begin{lem}\label{lem:5.3}
Let $\vrp: V\to W'$ be a $\lm$-linear map. Assume that $(q,b)$ and
$(q',b')$ are quadratic pairs on $V$ and $W'$, respectively.
Assume further that $(\veps_i\ds|i\in I)$ is a system of
generators of the $R$-module $V,$ and put
$\veps_i':=\vrp(\veps_i)$, $i\in I$. Assume finally that the
conditions \eqref{eq:5.1} and~\eqref{eq:5.2} above hold for
$x,y$ in the family $(\veps_i\ds|i \in I),$ i.e.,
\begin{equation}\label{eq:5.3}
q'(\veps_i')=\lm(q(\veps_i)),
\end{equation}
\begin{equation}\label{eq:5.4}
b'(\veps_i',\veps_j')=\lm(b(\veps_i,\veps_j))
\end{equation}
for all $i,j\in I.$ Then $\vrp$ is $\lm$-isometric with respect to
$(q,b)$ and $(q',b').$
\end{lem}

\begin{proof}
Given $x,y\in V$ we choose presentations
$$x=\sum_{i\in I}x_i\veps_i,\qquad y=\sum_{i\in I}y_i\veps_i$$
with $x_i,y_i\in R.$ Then
$$\vrp(x)=\sum_{i\in I}\lm(x_i)\veps_i',\qquad \vrp(y)=\sum_{i\in I}\lm(y_i)\veps_i'.$$
We  have (cf. \eqref{eq:1.9})
$$q(x)=\sum_{i\in I}x_i^2q(\veps_i)+\sum_{i<j} x_ix_jb(\veps_i,\veps_j),$$
$$q'(\vrp(x))=\sum_{i\in I}\lm(x_i)^2q'(\veps_i')+\sum_{i<j}\lm(x_i)\lm(x_j)
b'(\veps_i',\veps_j').$$

Applying $\lm$ to the first equation and using \eqref{eq:5.3} and
\eqref{eq:5.4}, we obtain
$$\lm(q(x))=q'(\vrp(x)).$$
We further have
$$b(x,y)=\sum_{i,j\in I} x_iy_jb(\veps_i,\veps_j),$$
$$b'(\vrp(x),\vrp(y))=\sum_{i,j\in I}\lm(x_i)\lm(y_j)b'(\veps_i',\veps_j').$$
By use of \eqref{eq:5.4} we obtain that
$$\lm(b(x,y))=b'(\vrp(x),\vrp(y)),$$
as desired.
\end{proof}

We are ready to characterize the $\lm$-extension of an expansive
quadratic pair (Definition~\ref{defn:4.15}) by a universal
property.

\begin{thm}\label{thm:5.4}
Assume that $(q,b)$ is an expansive quadratic pair on an
$R$-module $V.$ Assume further that $(q',b')$ is  a quadratic pair
on an $R'$-module $W'$ and that $\vrp:V\to W'$ is a
$\lm$-isometric map with respect to $(q,b)$ and $(q',b').$ \{In
particular, $\vrp$ is $\lm$-linear.\} Then the unique $R'$-linear
map $\psi: V_{R'}\to W'$ with $\vrp=\psi\circ \rho$ (where
$\rho:=\rho_V^\lm)$ is isometric with respect to $(q_{R'},b_{R'})$
and $(q',b').$
\end{thm}

\begin{proof}
We know from above that $\rho$ is isometric with respect to
$(q,b)$ and $(q',b').$ We choose a system of generators
$(\veps_i\ds|i\in I)$ of $V$ and then also have  the system
$(\veps_i'\ds|i\in I)$ of generators  of $V_{R'}$ with
$\veps_i':=\rho(\veps_i)$. It follows that
$$q_{R'}(\veps_i')=\lm(q(\veps_i))=q'(\vrp(\veps_i))=q'(\psi(\veps_i'))$$
for all
$i,j\in I$, and
$$b_{R'}(\veps_i',\veps_j')=\lm b(\veps_i,\veps_j)=b'(\vrp(\veps_i),\vrp(\veps_j))=b'(\psi(\veps_i'),\psi
(\veps_j')).$$ Thus, by \lemref{lem:5.3},
$\psi:V_{R'}\to W$ is isometric with respect to $(q_{R'},b_{R'})$
and~ $(q',b').$
\end{proof}

In a similar way we characterize the $\lm$-scalar extension of a
quadratic form on an $R$-module (instead of a quadratic pair) by a
universal property.

\begin{thm}\label{thm:5.5}
Assume that $q$ is a quadratic form on an $R$-module $V$ and $q'$
is an expansive quadratic form on an $R$-module $W'.$ Assume
further that $\vrp: V\to W'$ is a $\lm$-linear map which is
$\lm$-isometric with respect to $q$ and $q'.$ Let
$\psi:V_{R'}\to W'$ be the unique $R'$-linear map with
$\psi\circ\rho_V^\lm=\vrp.$ Then $q$ is expansive and $\psi$ is
isometric with respect to $q_{R'}$ and $q'.$
\end{thm}

\begin{proof}
We choose an expansion $B':W'\times W'\to R'$ of $q'.$ Let
$b':=B'+(B')^\ont.$ Then $(q',b')$ is a quadratic pair with
expansion $B'.$ We introduce the bilinear forms
$B:=B'\circ(\vrp\times\vrp)$ and $b:=b'\circ(\vrp\times\vrp)$ on
$V.$ Then $(q,b)$ is a quadratic pair on $V$ with expansion $B,$
and $\vrp$ is $\lm$-isometric with respect to $(q,b)$ and
$(q',b').$ It follows by \thmref{thm:5.4} that $\psi$ is isometric
with respect to $(q_{R'},b_{R'})$ and $(q',b')$. A fortiori $\psi$
is isometric with respect to $q_{R'}$ and $q'.$
\end{proof}

\begin{cor}\label{cor:5.6}
Let $\vrp: U\to V$ be an $R$-linear map. Assume that $(q_1,b_1)$
is a quadratic pair on $U$ and $(q,b)$ is an expansive quadratic
pair on $V.$
\begin{enumerate} \ealph \dispace
\item If $\vrp$ is isometric with respect to $q_1$ and $q,$
then $q_1$ is expansive and $\vrp_{R'}:U_{R'}\to V_{R'}$ is
isometric with respect to $(q_1)_{R'}$ and $q_{R'}.$
    \item If $\vrp$ is isometric with respect to $(q_1,b_1)$ and $(q,b)$, then $(q_1,b_1)$ is expansive and $\vrp_{R'}$ is isometric with respect to $((q_1)_{R'},(b_1)_{R'})$ and $(q_{R'},b_{R'}).$
        \end{enumerate}
        \end{cor}

        \begin{proof} (a): Let $B:V\times V\to R$ be an expansion of $q.$ Then $B\circ(\vrp\times\vrp)$ is an expansion of $q_1$ and $B_{R'}$ is an expansion of $q_{R'} $
        (essentially by construction of $q_{R'}).$ The $\lm$-linear map $\rho_V\circ\vrp$ is isometric with respect to $q_1$ and $q_{R'},$ and $\rho_V\circ \vrp= \vrp_{R'}\circ\rho_U$ (cf. \corref{cor:4.7}). We conclude by \thmref{thm:5.5} that $\vrp_{R'} $ is isometric with respect to $(q_1)_{R'}$ and $q_{R'}.$ \pSkip
        (b): An analogous, perhaps even easier argument using \thmref{thm:5.4}.
        \end{proof}

         \corref{cor:5.6} can also be deduced  from \propref{prop:4.13} in a straightforward way.
\pSkip
        It now makes sense to expand the terminology of subordinate quadratic forms and pairs from \S\ref{sec:3} (cf. Definition \ref{defn:3.3}) as follows.

        \begin{defn}\label{defn:5.7} $ $

        \begin{enumerate} \ealph  \dispace
\item Given quadratic forms $q_1 : U\to R$ and $q':V'\to R'$,
we say that $q_1$ is $\lm$-\bfem{subordinate to} $q,$ if there
exists a map $\vrp: U\to V'$ which is $\lm$-isometric with respect
to $q_1$ and $q'$ (in particular $\vrp$ is $\lm$-linear). We then
write $q\prec_\lm q'.$ More elaborately, we sometimes say,  that
$q_1$ is $\lm$-\bfem{subordinate to} $q$ \bfem{via} $\vrp$, and
write $q_1\prec_{\lm,\vrp}q'.$

    \item We use a completely analogous terminology for quadratic pairs $(q_1,b_1)$ on $U$ and $(q',b')$ on $V,$ and write $(q_1,b_1)\prec_\lm(q',b')$ or $(q_1,b_1)\prec_{\lm,\vrp}(q',b').$
        \end{enumerate}
        \end{defn}

        Our main reason for this terminology (and that in  Definition \ref{defn:3.3} which covers the case $\lm=\id_R)$ is that it allows to pose conveniently one of the basic general problems of quadratic forms over semirings.

        \begin{problem}\label{prob:5.8}
        Assume that $q_1$ is a quadratic form on an $R$-module $U$ and $q'$ is a quadratic form on an $R'$-module $V'.$ Determine whether $q_1\prec_\lm q'$ or not.
        \end{problem}

As a consequence of Example \ref{examp:5.2}, Theorems
\ref{thm:5.4} and \ref{thm:5.5}, and \corref{cor:5.6},  we state
some facts about $\lm$-subordinate quadratic forms and pairs,
which are now obvious.

\begin{schol}\label{schol:5.9}
$ $

\begin{enumerate} \ealph  \dispace
\item Assume that $q_1$ is a quadratic form over $R$ (i.e., on
some $R$-module), and $q'$ is a quadratic form over $R'.$ If
$q_1\prec_\lm q',$ and $q'$ is expansive, then $q_1$ is expansive
and $(q_1)_{R'}\prec q'.$ Conversely, if $q_1$ is expansive and
$(q_1)_{R'}\prec q'$, then $q_1\prec_\lm q'.$

    \item The same remains true if we replace $q_1$ and $q'$ by quadratic pairs $(q_1,b_1)$ and $(q',b')$ over $R$ and $R',$ respectively.

        \item If $q_1 $ and $q$ are expansive quadratic forms over $R$ with $q_1\prec q,$ then $(q_1)_{R'}\prec q_{R'}.$

            \item If $(q_1,b_1)$ and $(q,b)$ are expansive quadratic pairs over $R,$ and $(q_1,b_1)\prec (q,b),$ then $((q_1)_{R'},(b_1)_{R'})\prec(q_{R'},b_{R'}).$
                \end{enumerate}
                \end{schol}

\section{Pushing down quadratic pairs}\label{sec:6}

In all the following, $R$ and $R'$ are semirings and $\lm :R\to
R'$ is a \textit{surjective} semiring homomorphism.

Assume
that $V$ is an $R$-module, $V'$ an $R'$-module,  $\vrp : V\to
V'$  a \textit{surjective} $\lm $-linear map, and
 $(q,b)$  a quadratic pair on $V.$ Then we may ask whether
there exists a quadratic pair  $(q',b')$ on $V'$ such that
$\vrp $ is $\lm $-isometric with respect to $(q,b)$ and
$(q',b')$, i.e., for all $x,y\in V,$
\begin{equation}\label{eq:6.1}
q'\circ\vrp =\lm \circ q,\qquad
b'\circ(\vrp \times\vrp )=\lm \circ b.
\end{equation}
 Since $\vrp $ is surjective,  $(q',b')$ is then uniquely determined by $(q,b).$

 \begin{defn}\label{defn:6.1}
  When \eqref{eq:6.1} holds, we call $(q',b')$ a $\lm $-\bfem{pushdown} of the quadratic pair ~ $(q,b)$, more precisely, the $\lm $-\bfem{pushdown   of} $(q,b)$ \bfem{along}~$\vrp $, and say that  the surjective $\lm $-linear map~$\vrp $ is a $\lm $-\bfem{pushdowner} of the pair $(q,b).$ In the important case that $R=R'$ and $\lm =\id_R,$ we also say that $(q',b')$ is a \bfem{compression} of $(q,b)$ and that the $R$-linear map $\vrp $ is a  \bfem{compressor} of $(q,b).$
 \end{defn}

\begin{comment*}
We already encountered this situation  in \S\ref{sec:1}: $(q',b')$
is a $\lm $-pushdown of $(q,b)$ along~$\vrp $ iff $(q,b)$ is
a $\lm $-lift of $(q',b')$ via $\vrp $, cf. Definition
\ref{defn:1.10}. But the perspective has changed. While $(q',b')$
can have at most one $\lm $-pushdown along $\vrp ,$ the
$\lm $-lift $(q',b')$, if it exists, most often is not uniquely
determined by $\vrp .$
\end{comment*}

\begin{example}\label{examp:6.2}
Assume that $(q,b)$ is an expansive quadratic pair on $V.$ So we
have the scalar extension $(q_{R'},b_{R'})$ with respect to
$\lm: R \twoheadrightarrow R'$ at our disposal, and
$\rho=\rho_V^\lm :V\to V_{R'}$ is isometric with respect to
$(q,b)$ and $(q_{R'},b_{R'})$, cf. \S\ref{sec:5}. Since $\lm $
is surjective, the $\lm $-linear map~ $\rho$ is again
surjective, as follows from the fact that $\rho(V)$ generates
$V_{R'},$ cf. \thmref{thm:4.5}. Thus $(q_{R'},b_{R'})$ is a
$\lm $-pushdown along $\rho.$ If $\vrp : V\to V'$ is a
surjective $\lm $-linear map, then we have a unique
factorisation $\vrp =\psi\circ\rho$ with $\psi:V_{R'}\to V'$ a
surjective $R'$-linear map. If in addition $\vrp $ is isometric
with respect to $(q,b)$ and a quadratic pair $(q',b')$ on $V',$  we
learn from \thmref{thm:5.4} that $(q',b')$ is the compression of
$(q_{R'},b_{R'})$ along $\psi.$
\end{example}

Below in this section we shall  see that essentially the same
situation as in Example \ref{examp:6.2} occurs for
\textit{any} quadratic pair  $(q,b)$ on $V,$ not necessarily expansive.

\pSkip

For later use, we state an easy lemma concerning transitivity of
pushdowns.

\begin{lemma}\label{lem:6.3}
As before, $\lm :R\to R'$ is surjective and $\vrp :V\to V'$
is a surjective $\lm $-linear map. Assume that $\mu:R'\to R''$
is a second surjective semiring homomorphism and $\psi: V'\to T'$
is a surjective $\mu$-linear map. Let $(q,b)$ be a quadratic pair
on $V.$
\begin{enumerate} \ealph \dispace
\item If $(q,b)$ has a $\lm $-pushdown $(q',b')$ along
$\vrp $ and $(q',b')$ has a $\mu$-pushdown $(q'',b'')$ along~
$\psi,$ then $(q'',b'')$ is a $(\mu\lm) $-pushdown of  $(q,b)$ along
$\psi\vrp .$

    \item Assume that $(q,b)$ has a $\lm $-pushdown $(q',b')$ along $\vrp $ and a $(\mu\lm)$-pushdown $(q'',b'')$ along~$\psi.$ Then $(q'',b'')$ is a $\mu$-pushdown of $(q',b')$ along $\psi.$

        \item Assume that $\mu$ is also injective,  hence a semiring isomorphism.
        If $(q,b)$ has a $ (\mu\lm )$-pushdown $(q'',b'')$ along~$\psi\vrp $, then $(q,b)$ has a $\lm $-pushdown $(q',b')$ along $\vrp $, and $(q'',b'')$ is a $\mu$-pushdown of $(q',b')$ along $\psi.$
        \end{enumerate}
        \end{lemma}

        \begin{proof} (a): A straightforward check. \pSkip
        (b): From $q'\vrp =\lm  q,$ $b'(\vrp \times\vrp )=\lm  b,$ and $q''\psi\vrp =\mu\lm  q,$ $b''(\psi\times\psi)(\vrp \times\vrp )=\mu\lm  b,$ we conclude that
        $$q''\psi\vrp =\mu q'\vrp ,\qquad b''(\psi\times\psi)(\vrp \times\vrp )=\mu b'(\vrp \times\vrp ).$$
        Since $\vrp $ is surjective, this implies
        $$q''\psi=\mu q',\qquad b''(\psi\times\psi)=\mu b',$$
        as desired. \pSkip
        (c): We have $q''\psi\vrp =\mu\lm  q,$\, $b''(\psi\times\psi)(\vrp \times\vrp )=\mu\lm  b.$\newline
        Let $x_1,x_2,y_1,y_2\in V,$ and assume that $\vrp (x_1)=\vrp (x_2),$ $\vrp (y_1)=\vrp (y_2).$ Then
        $$\mu\lm  q(x_1)=q''\psi\vrp (x_1)=q''\psi\vrp (x_2)=\mu\lm  q(x_2).$$
        Since $\mu$ is injective, this implies that $\lm  q(x_1)=\lm  q(x_2)$. In the same vein, we obtain that $\lm  b(x_1,y_1)=\lm  b(x_2,y_2).$ Thus we have well-defined maps $$q':V'\To R', \qquad b':V'\times V'\To R'$$ with $q'(\vrp (x))=\lm  q(x)$ and $b'(\vrp (x),\vrp (y))=\lm  b(x,y) $ for all $x,y\in V.$ By a straightforward check we see that $(q',b')$ is a quadratic pair on $V',$ which by definition is a $\lm $-pushdown of $(q,b)$ along $\vrp .$ By part (b) it follows that $(q'',b'')$ is a $\mu$-pushdown of $(q',b').$
       \end{proof}

        We next construct $\lm $-pushdowns for any quadratic pair over $R$ by using appropriate equivalence relations. In a first step, we complete the theory of $\lm $-scalar extensions from \S\ref{sec:4} to the present case where $\lm $ is surjective.

\begin{defn}\label{defn:6.4}
As before, $\lm :R\to R'$ is surjective and $V$ is an
$R$-module. An equivalence relation $E$ on $V$ is called
\textbf{$\lm $-linear}, if $E$ is $R$-linear (cf. Definition
\ref{defn:4.3}) and
$$\lm (c_1)=\lm (c_2) \dss \Rightarrow c_1x\sim_Ec_2x$$
for all $c_1,c_2\in R$ and $x\in V.$ When this holds, we equip the
set $V/E$ with the structure of an $R'$-module by the rules
$(x,y\in E,\, c\in R)$
\begin{align*} [x]_E+[y]_E  :=[x+y]_E,\quad
\lm (c)\cdot[x]_E  :=[cx]_E. \end{align*}
\end{defn}

Note that this is the unique $R'$-module structure on the set
$V/E$ for which the map $\pi_E:V\twoheadrightarrow V/E$ is
$\lm $-linear.

\begin{remarks}\label{rem:6.5}
{}\quad

\begin{enumerate} \dispace \eroman
  \item If $\vrp :V\to V'$ is a $\lm $-linear map, then the
relation $E(\vrp )$ on the set $V$ (defined by
$x\sim_{E(\vrp )}y \Leftrightarrow
\vrp (x)=\vrp (y)$)
is
$\lm $-linear and we have a well-defined $R'$-linear map
$$\overline\vrp :V/E(\vrp )\To V',$$  given for $x\in V$ by
\begin{equation}\label{eq:6.2}
\overline\vrp ([x]_{E(\vrp )}):=\vrp (x),
\end{equation}
 which means that
\begin{equation}\label{eq:6.3}
\vrp =\overline\vrp \circ\pi_{E(\vrp )}.
\end{equation}
If $\vrp $ is surjective, then $\overline\vrp $ is an
$R'$-linear isomorphism.

\item If $\vrp :V\to V'$,  $\psi: V\to V''$ are $\lm $-linear
maps and $\vrp $ is surjective, then $\psi$ factors through
$\vrp $ iff $E(\vrp )\subset E(\psi).$

\end{enumerate}
\end{remarks}

In all the following, $(q,b)$ is a quadratic pair on an $R$-module
$V.$

\begin{defn}\label{defn:6.6}
We call an equivalence relation $E$ on the set $V$ a
$\lm $-\bfem{pushdown equivalence} (abbreviated
$\lm $-\bfem{pd-equivalence}) \bfem{for} $q,b$, if $E$ is
$\lm $-linear and the surjective $\lm $-linear map
$\pi_E:V\twoheadrightarrow V/E$ pushes $(q,b)$ down to a quadratic
pair on $V/E$, which we denote by $(q/E,b/E).$ In the case that
$\lm =\id_R$, we also  say that $E$ is a
\bfem{pd-equivalence relation}, or a \bfem{compressive equivalence
relation for} $(q,b)$.
\end{defn}

\begin{remarks}\label{rem:6.7}
{}\quad
\begin{enumerate} \eroman \dispace
  \item  If $E$ is a $\lm $-pd-equivalence for $(q,b)$, then clearly
for any $x,y\in V$ the following holds:
\begin{equation}\label{eq:6.4}
(q/E)([x]_E)=\lm (q(x)),
\end{equation}
\begin{equation}\label{eq:6.5}
(b/E)([x]_E,[y]_E)=\lm (b(x,y)).
\end{equation}

\item If $\vrp :V\to V'$ is a surjective $\lm $-linear map and
$(q,b)$ has a $\lm $-pushdown $(q',b')$ along ~$\vrp $, then
it is clear from \eqref{eq:6.2}, \eqref{eq:6.3}, \eqref{eq:6.4},
\eqref{eq:6.5}, that $E(\vrp )$ is a $\lm $-pd-equivalence for
$(q,b)$, and that the $R'$-linear isomorphism
$\overline\vrp :V/E(\vrp )\to V'$ induced by $\vrp $ is
isometric with respect to $(q/E(\vrp ),b/E(\vrp ))$ and
$(q',b').$
\end{enumerate}

\end{remarks}

\begin{defn}\label{defn:6.8}
We define an equivalence relation $E:=E_\lm (q,b)$ on the set
$V$ by declaring for $x_1,x_2\in V$ that $x_1\sim_E x_2$ iff
$\lm (q(x_1))=\lm (q(x_2))$ and
$\lm (b(x_1,y))=\lm (b(x_2,y))$ for all $y\in V.$ We denote
the associated map $\pi_E:V\to V/E$ by $\pi_{q,b}^\lm .$ In the
  case that $\lm =\id_R$, we usually write $E(q,b)$
and $\pi_{q,b}$ instead of $E_\lm (q,b)$ and
$\pi_{q,b}^\lm .$
\end{defn}

\begin{thm}\label{thm:6.9}
$E:=E_\lm (q,b)$ is a $\lm $-pd-equivalence for $(q,b)$.
Thus the map $\pi_E$ is a surjective $\lm $-linear map from $V$
to the  $R'$-module $\overline V:=V/E$, which pushes $(q,b)$ down to
a quadratic pair $(\htq _\lm ,\htb _\lm ):=(q/E,b/E)$ on
$\overline V.$
\end{thm}

\begin{proof} An easy check reveals that $x_1\sim_Ex_2$ implies  $cx_1\sim_Ecx_2$ for every $c\in R,$ and that $\lm (c_1)=\lm (c_2)$ implies  $c_1 x\sim_Ec_2x$ for every $x\in R.$ If $x_1\sim_Ex_2$, then for any $y,z\in V$ we have
\begin{align*} \lm  b(x_1+y,z) & =\lm  b(x_1,z)+\lm  b(y,z) \\ & =\lm  b(x_2,z)+\lm  b(y,z)
=\lm  b(x_2+y,z)\end{align*} and
\begin{align*} \lm  q(x_1+y) & =\lm  q(x_1)+\lm  b(x_1,y)+\lm
q(y) \\ & =\lm  q(x_2)+\lm  b(x_2,y)+\lm  q(y)  =\lm
q(x_2+y).\end{align*} Thus $x_1\sim_E x_2$ implies $x_1+y\sim_Ex_2+y$ for
any $y\in V.$ This proves that $E$ is $\lm $-linear (cf.
Definitions \ref{defn:4.3} and \ref{defn:6.4}). If $x_1\sim_E x_2$
and $y_1\sim_E y_2,$ then
$$\lm  b(x_1,y_1)=\lm  b(x_2,y_1)=\lm  b(y_1,x_2)=\lm
b(y_2,x_2)=\lm  b(x_2,y_2).$$ Thus we have a well-defined map
$\htb _\lm :V/E\times V/E\to R'$ given by the rule
\eqref{eq:6.5}. We also have a well-defined map $\htq_\lm :V/E\to R'$ given by the rule \eqref{eq:6.4}. Using the
fact that $\lm :R\to R'$ is surjective, it is  easily
verified that $\htb _\lm $ is a symmetric bilinear form and
$\htq _\lm $ is a quadratic form with companion $\htb_\lm .$ Formulas \eqref{eq:6.4} and \eqref{eq:6.5} say
that $\pi_E$ is $\lm $ isometric with respect to $(q,b)$ and
$(\htq _\lm ,\htb _\lm )$, i.e., $(\htq _\lm ,\htb_\lm )$ is a $\lm $-pushdown of $(q,b),$ along $\pi_E.$
\end{proof}

\begin{thm}\label{thm:6.10}
Assume that $\vrp :V\twoheadrightarrow V'$ is a surjective
$\lm $-linear map. The following are equivalent:
\begin{enumerate} \dispace \eroman
\item  $\vrp : V\twoheadrightarrow V'$ is a
$\lm $-pushdown of $(q,b)$,

\item $E(\vrp )\subset
E_\lm (q,b).$
\end{enumerate}
If (i) and  (ii) hold, then the $R'$-linear map $\psi: V'\to
V/E_\lm (q,b)$ induced by $\vrp  $ (i.e.,
$\pi_{q,b}^\lm =\psi\circ\vrp )$ compresses the
$\lm $-pushdown $(q',b')$ of $(q,b)$ along $\vrp $ to the
$\lm $-pushdown $(\htq _\lm ,\htb _\lm )$ of $(q,b)$
along~$\pi_{q,b}^\lm .$
\end{thm}

\begin{proof}
(i)$\Rightarrow$(ii): Assume that $(q,b)$ has a $\lm $-pushdown $(q',b')$ along
$\vrp .$ Let $(x_1,x_2)\in E(\vrp ),$ i.e., $x_1,x_2\in V$ and
$\vrp (x_1)=\vrp (x_2).$ Then
$$\lm  q(x_1)=q'\vrp (x_1)=q'(x_2)=\lm q(x_2).$$
Likewise, for any $y\in V,$
$$\lm
b(x_1,y)=b'(\vrp (x_1),\vrp (y))=b'(\vrp (x_2),\vrp (y))=\lm
b(x_2,y).$$ Thus $(x_1,x_2)\in E_\lm (q,b).$ This proves that
$E(\vrp )\subset E_\lm (q,b).$ \pSkip
(ii)$\Rightarrow$(i): Assume that $E(\vrp )\subset E_\lm (q,b).$ Then
$\pi_{q,b}^\lm $ has a factorization
$\pi_{q,b}^\lm =\psi\circ\vrp $ with $$\psi:V'\To \overline
V:=V/E_\lm (q,b)$$ a surjective $R'$-linear map (cf. Remark
\ref{rem:6.5}.(ii)). We conclude by \lemref{lem:6.3}.(c) that $(q,b)$
has a $\lm $-pushdown $(q',b')$ along $\vrp $,  and that $(q',b')$
is compressed to $(\htq _\lm ,\htb _\lm )$ along $\psi$.
\end{proof}

\begin{cor}\label{cor:6.11}
{}\quad

\begin{enumerate} \dispace \ealph
\item  Every quadratic pair $(q,b)$ on $V$ has a
$\lm $-pushdown $(q_\lm ,b_\lm )$ along the
$\lm $-scalar extension $\rho=\rho^\lm _V:V\to V_\lm $ of
$V.$

\item If $(q,b)$ has a $\lm $-pushdown $(q',b')$ along a
given surjective $\lm $-linear map $\vrp:V \to V'$, then the $R'$-linear map
$\overline\vrp :V_\lm \to V'$  induced by $\vrp $ (i.e.,
$\vrp =\overline\vrp \circ\rho)$ compresses
$(q_\lm ,b_\lm )$ to $(q',b').$

\item In particular, the $R'$-linear map
$$\overline\pi_{q,b}^\lm : V_\lm \To V/E_\lm (q,b)$$
induced by $\pi_{q,b}^\lm :V\to V/E_\lm (q,b)$ compresses
$(q_\lm ,b_\lm )$ to the pair $(\htq _\lm ,\htb_\lm )$ appearing in Theorems \ref{thm:6.9} and~
\ref{thm:6.10}.
\end{enumerate}
\end{cor}

\begin{proof}
(a): Since the map $\pi_{q,b}^\lm :V\twoheadrightarrow
V/E_\lm (q,b)$ is $\lm $-linear, we have a factorization
$\pi_{q,b}^\lm =\psi\circ\rho_V^\lm $ with
$\psi:V_\lm \to V/E_\lm (q,b)$ an $R'$-linear map. Thus
$$E(\rho_V^\lm )\subset E(\pi_{q,b}^\lm )=E_\lm (q,b).$$
By \thmref{thm:6.10} we conclude that $(q,b)$ has the
$\lm $-pushdown $(q_\lm ,b_\lm )$ along $\rho_V^\lm .$
\pSkip
(b): We have a factorization $\vrp =\psi\circ\rho_V^\lm $ with
$\psi:V_\lm \to V'$ a surjective $R'$-linear map. We conclude,
again by \lemref{lem:6.3}.(c), that $\psi$ compresses
$(q_\lm ,b_\lm )$ to $(q',b').$
\pSkip
(c): A special case of (b).
\end{proof}

\begin{defn}\label{defn:6.12}
We call $(q_\lm ,b_\lm )$ the $\lm $-\bfem{scalar
extension of the pair} $(q,b)$. When  it is clear from the context,
which homomorphism $\lm :R \twoheadrightarrow R'$ is under
consideration, we  write $(q_{R'},b_{R'})$ instead of
$(q_\lm ,b_\lm ).$
\end{defn}

\begin{remark}\label{rem:6.13}
Assuming  that the pair $(q,b)$ is expansive, we have already
obtained in \S4 a $\lm $-scalar extension
$(q_\lm ,b_\lm )_{\old}$ on $V_\lm $ (Definition
\ref{defn:4.15}). By \thmref{thm:5.4}, the unique map $\psi:
V_\lm \to V_\lm $ with
$\psi\circ\rho_V^\lm =\rho_V^\lm $ is isometric with respect
to $(q_\lm ,b_\lm )_{\old}$ and the defined
$\lm $-scalar extension $(q_\lm ,b_\lm )_{\new}.$
Since $\rho_V^\lm $ is surjective, $\psi$ is the identity of
$V_\lm $. We conclude that
$(q_\lm ,b_\lm )_{\new}=(q_\lm ,b_\lm )_{\old}.$
\end{remark}

Thus our present terminology coincides  with Definition
\ref{defn:4.15}. \pSkip

We next spell out much of the contents of Theorems
\ref{thm:6.9}, \ref{thm:6.10} and \corref{cor:6.11} in another,
more condensed, way and draw some consequences.

\begin{defn}\label{defn:6.14}
Let $\vrp _1:V\to V'_1$ and $\vrp _2:V\to V_2'$ be surjective
$\lm $-linear maps. We say that $\vrp _2$ \bfem{dominates}
$\vrp _1$ and write $\vrp _1\preceq \vrp _2,$ if there
exists a map $\alpha: V_1'\to V_2'$ for which
$\alpha\circ\vrp _1=\vrp _2$ (i.e., $\vrp _2$ factors
through $\vrp _1).$ This map $\vrp $ is
surjective and $R'$-linear. We call~ $\vrp _1$ and $\vrp _2$
\bfem{equivalent}, and write $\vrp_1\cong \vrp _2$,  if
$\vrp _1\preceq \vrp _2$ and $\vrp _2\preceq \vrp _1.$
Then $\vrp _2=\alpha\circ\vrp _1$ with $\alpha$ an
$R'$-linear isomorphism from $V_1'$ to $V_2'.$
\end{defn}

\begin{remark}\label{rem:6.15}
If  both $\vrp _1$ and $\vrp _2$ are
$\lm $-pushdowners of $(q,b)$ to quadratic pairs $(q_1',b_1')$
and $(q_2',b_2')$ and $\vrp _1\preceq \vrp _2$, then the
$R'$-linear surjective map $\alpha: V_1'\twoheadrightarrow V_2'$
is isometric with respect to $(q_1',b_1')$ and $(q_2',b_2')$ and
thus $\alpha$ compresses $(q_1',b_1')$ to $(q_2',b_2').$
\end{remark}

\begin{notations*}\label{notat:6.16}  $ $
\begin{enumerate} \dispace \ealph

\item  We denote the equivalence class of a $\lm $-linear
surjection $\vrp $ on $V$ by $[\vrp ]$ and the set of all
these classes by $\LIN_\lm (\vrp ).$ The set
$\LIN_\lm (\vrp )$ is partially ordered in the obvious
way,
$$[\vrp _1]\le[\vrp _2] \dss \Leftrightarrow
\vrp _1\preceq\vrp _2.$$

\item On the other hand, we have the poset\footnote{= ``partially
ordered set''} $\ELIN_\lm (V)$ consisting of all
$\lm $-linear equivalence relations on $V,$ ordered by
inclusion
$$E_1\le E_2 \dss \Leftrightarrow E_1\subset E_2.\ \footnote{Recall that
 we regard every equivalence relation on $V$ as a subset of
 $V\times V.$}
$$

\end{enumerate}
\end{notations*}

\begin{remark}\label{rem:6.17}
The poset $\ELIN_\lm (\vrp )$ is a complete lattice,
more precisely a complete sublattice of the lattice of all
equivalence relations on the set $V,$ as is easily seen.  We have
a poset isomorphism
\begin{equation}\label{eq:6.6}
E_V^\lm :\LIN_\lm (V)\Isoto
\ELIN_\lm (V),\qquad [\vrp ]\mapsto E(\vrp ),
\end{equation}
with inverse map $E\mapsto [\pi_E].$ Thus also
$\LIN_\lm (V)$ is a complete lattice. It has  the bottom
element $[\rho_V^\lm ]$ and the top element $[0]$, represented
by the zero map $V\to\{0\}.$
\end{remark}

\begin{notations*}\label{notations:6.18}
We denote the set of equivalence classes
$[\vrp ]\in\LIN_\lm (V)$ of $\lm $-pushdowners
$\vrp $ of $(q,b)$ by $\PD_\lm (q,b)$ and the set of
$\lm $-\text{pd}-equivalences for $(q,b)$ by
$\EPD_\lm (q,b),$  written respectively
$\PD(q,b)$ and $\EPD(q,b)$ in the case
$\lm =\id_R.$ We regard the sets $\PD_\lm (q,b)$ and
$\EPD_\lm (q,b)$ as subposets of $\LIN_\lm (V)$
and $\ELIN_\lm (U)$, respectively.
\end{notations*}

The following is now a direct consequence of Theorems
\ref{thm:6.9}, \ref{thm:6.10} and \corref{cor:6.11}.

\begin{schol}\label{schol:6.19}
As before $(q,b)$ is a quadratic pair on an $R$-module $V$ and
$\lm :R\twoheadrightarrow R'$ is a surjective semiring
homomorphism.
\begin{enumerate} \ealph \dispace

 \item $\EPD_\lm (q,b)$ is the set of all
 $\lm $-linear equivalence relations $E$ on $V$ with
 $$E(\rho_V^\lm )\subset E\subset E_\lm (q,b).$$
 Thus $\EPD_\lm (q,b)$ is a complete sublattice and a
 lower subset of $\ELIN_\lm (V)$ with bottom
 $E(\rho_V^\lm )$\footnote{the same for all $(q,b)$ on $V$!}
 and top $E_\lm (q,b).$

 \item We have an isomorphism of posets
 \begin{equation}\label{eq:6.7}
 \EPD_\lm (q,b) \Isoto \PD_\lm (q,b),\qquad
 E\mapsto [\pi_E],
 \end{equation}
 with inverse map $[\vrp ]\mapsto E(\vrp ).$

 \item We also have an isomorphism of posets
\begin{equation}\label{eq:6.8}
 \PD(q_\lm ,b_\lm ) \Isoto  \PD_\lm (q,b),\qquad
[\psi]\mapsto [\psi\circ\rho_V^\lm ].
 \end{equation}

\end{enumerate}
\end{schol}

\begin{defn}\label{defn:6.20} $ $
\begin{enumerate} \ealph \dispace
  \item Recall that in the case that $\lm = \id_R$ we denote the top element
$E_\lm (q,b)$ of the poset
$$\EPD(q,b):=\EPD_\lm (q,b)$$ by $E(q,b)$ and the map
$\pi_{E(q,b)}$ by $\pi_{q,b}.$ Consequently,  the particular
compression $(q/E(q,b),\, b/E(q,b))$ of $(q,b)$ along $\pi_{q,b}$ is called
the \bfem{terminal compression} of $(q,b)$ and is denoted by~
$(\htq ,\htb )$.

\item This notation is in harmony with the definition
$$(\htq _\lm ,\htb _\lm ):=(q/E_\lm (a,b),\,
b/E_\lm (a,b))$$ appearing above, since  $(\htq_\lm ,\htb _\lm )$ is indeed the terminal compression of the pair
$(q_\lm ,b_\lm )$ (\corref{cor:6.11}.(c)). We call $(\htq_\lm ,\htb _\lm )$ the \bfem{terminal}
$\lm $-\bfem{pushdown} of $(q,b),$ while
$(q_\lm ,b_\lm )$ is termed the \bfem{initial}
$\lm $-\bfem{pushdown} of $(q,b).$\footnote{The notation
$q_\lm $ is slightly abusive, since this quadratic form not only
depends on $q$ and $\lm $ but also on the companion $b$ of
$q.$}

\item If $(q,b)=(\htq ,\htb ),$ i.e., $E(q,b)$ is the diagonal of
$V\times V$, then we say that $(q,b)$ is \bfem{incompressible}.
\end{enumerate}

\end{defn}

\begin{prop}\label{prop:6.21}
Let $\mu: R'\to R''$ be a second surjective semiring homomorphism.
Assume that $(q',b')$ is the $\lm $-pushdown of $(q,b)$ along
some surjective $\lm $-linear map $\vrp : V\to V'.$ Then we
have a natural isometry
\begin{equation}\label{eq:6.9}
 (\htq'_\mu,\htb'_\mu)\cong( \htq_{\mu\lm },\htb_{\mu\lm }).
 \end{equation}
 \end{prop}

 \begin{proof}
 This follows from the fact that $(\htq _\mu',\htb _\mu')$ is a
 $(\mu\lm )$-pushdown of $(q,b)$, cf. \lemref{lem:6.3}.(a), and
 is incompressible.
 \end{proof}

\begin{defn}\label{def:6.20}
We call a quadratic pair $(q,b)$ on an $R$-module $V$, $R$ any semiring, \textbf{rigid}, if $q$
is rigid, and so $b$ is the unique companion of $q$.\end{defn}

Recall that the pair $(q,b)$ is balanced iff $b(x,x) = 2 q(x)$ for every $x \in V$. Also recall that, if ~$V$ happens to be free $R$-module, then every quadratic form on $V$ has a balanced companion.

Rigid balanced quadratic pairs behave very well under pushdowns.

\begin{thm}\label{thm:6.20}
Let $R'$ be a subsemiring  of a (commutative) ring $R''$. Assume that $\lm: R' \twoheadrightarrow R$ is a surjective semiring homomorphism from $R'$ to a second semiring $R$, and that $\vrp:V' \twoheadrightarrow V$ is a surjective $\lm$-linear map from an $R'$-module $V'$ to an $R$-module $V$.
\begin{enumerate} \ealph
  \item If $(q',b')$ is a quadratic pair on $V'$, which can be pushed down via $(\vrp,\lm)$ to a pair $(q,b)$ on $V$, then both $(q',b')$ and $(q,b)$ are balanced and rigid.
  \item If $(q,b)$ is a balanced and rigid quadratic pair on $V$, and the $R'$-module $V'$ is free, then $(q,b)$ can be lifted via $(\vrp,\lm)$ to a pair $(q',b')$ on $V'$, and any such lift $(q',b')$ is balanced and rigid.
\end{enumerate}
\end{thm}

\begin{proof}
  (a) is just a reformulation of Proposition \ref{prop:1.9}, and (b) is then clear from (a) and Proposition \ref{prop:1.8}.
\end{proof}
\begin{schol}\label{schol:6.22}
  In consequence, if the terminal $\lm$-pushdown of $(q,b)$ via $(\vrp,\lm)$ is rigid and balanced, the same holds for all pushdowns of $(q,b)$. Conversely, if a rigid and balanced pair on an $R$-module $V$ is given, we have a pletora  of free $R'$-modules $V'$ at  hands with surjective $\lm$-linear maps $V'\twoheadrightarrow V$. Indeed, if $V$ is finitely generates by vectors $\veps_1, \dots, \veps_N$, we may choose~ $V'$ as a polynomial semiring $R'[t_1, \dots, t_N]$. Otherwise we can choose a polynomial semiring with infinitely many variables.
\end{schol}

\begin{lem}\label{lem:6.23} If $R$ is a semiring, for which the monoid $(R,+)$ is cancellative (i.e., $\forall a,b,c \in R: a+c = b+c \Rightarrow a =b$), then there exists a ring, which contains $R$ as a subsemiring.
\end{lem}

\begin{proof}
Very well known. We may take $T = R \times R / \sim$ with an equivalence relation, defined as
$(a_1,b_1) \sim (a_2, b_2) \Leftrightarrow a_1 + b_2 = a_2 + b_1$.
 \end{proof}

\begin{thm}\label{thm:6.24} Assume that $(q,b)$ is a rigid balanced quadratic pair on an $R$-module $V$ for $R$ a cancellative semiring, that $\lm: R_1 \to R$ is a semiring homomorphism, and that $\vrp: V_1 \to V$ is a $\lm$-linear map  from an $R_1$-module $V_1$ to $V$. Then every quadratic pair
$(q_1,b_1)$ on $V_1$, which is subordinate to $(q,b)$ via $(\vrp, \lm)$, is again rigid and balanced.
\end{thm}

\begin{proof}
  The restriction $(q|\vrp(V_1), b | \vrp(V_1))$ of $(q,b)$ to the submodule $\vrp(V_1)$ of $V$ is clearly balanced. It is also rigid, since $R$ embeds into a ring $T$ (Lemma \ref{lem:6.23}). We conclude, already by Example \ref{examp:1.9}, that $(q|\vrp(V_1), b | \vrp(V_1))$ is rigid. The map $\vrp$ is a $\lm$-isometry from $V_1$ onto $\vrp(V_1)$ with respect to $(q_1,b_1) $ and $(q|\vrp(V_1), b | \vrp(V_1))$. It follows by Theorem \ref{thm:6.20}.a that $(q_1, b_1)$ is rigid and balanced.
\end{proof}

The primary source of these considerations about rigid and balanced quadratic pairs has been Example \ref{examp:1.6}.

\end{document}